\documentclass[12pt,leqno]{article} 
\usepackage{amsmath,amssymb,amsthm,amsfonts, indentfirst}
\usepackage{enumerate,color,bm,graphicx}
\makeatletter
\def\@cite#1#2{[{{\bfseries #1}\if@tempswa , #2\fi}]}
\renewcommand{\section}{%
\@startsection{section}{1}{\z@}
{0.5truecm plus -1ex minus -.2ex}%
{1.0ex plus .2ex}{\bfseries\large}}
\def\@seccntformat#1{\csname the#1\endcsname.\ }
\makeatother

\numberwithin{equation}{section} 
\theoremstyle{theorem}
\newtheorem{thm}{Theorem}[section]

\newtheorem{lem}[thm]{Lemma}

\theoremstyle{definition}

\newtheorem{remark}{Remark}[section]

\newtheorem*{prth1.1}{Proof of Theorem 1.1 (existence part)}
\newtheorem*{prth1.1c}{Proof of Theorem 1.1 (continued)}
\newtheorem*{prth1.2}{Proof of Theorem 1.2}
\newtheorem*{prth1.3}{Proof of Theorem 1.3}

\newcommand{\pa}{\partial}
\newcommand{\Rn}{\mathbb{R}^n}

\newcommand{\Rone}{\mathbb{R}}

\newcommand{\tmax}{T_{\rm max}}

\begin{document}
\footnote[0]
    {2010{\it Mathematics Subject Classification}\/. 
    Primary: 35B44; Secondary: 35K55, 92C17.
    }
\begin{center}
    \Large{{\bf Blow-up in a quasilinear parabolic--elliptic Keller--Segel system with logistic source}}
\end{center}
\vspace{5pt}
\begin{center}
    Yuya Tanaka
   \footnote[0]{
    E-mail: 
    {\tt yuya.tns.6308@gmail.com}
    }\\
    \vspace{12pt}
    Department of Mathematics, 
    Tokyo University of Science\\
    1-3, Kagurazaka, Shinjuku-ku, 
    Tokyo 162-8601, Japan\\
    \vspace{2pt}
\end{center}
\begin{center}    
    \small \today
\end{center}

\vspace{2pt}
\newenvironment{summary}
{\vspace{.5\baselineskip}\begin{list}{}{%
     \setlength{\baselineskip}{0.85\baselineskip}
     \setlength{\topsep}{0pt}
     \setlength{\leftmargin}{12mm}
     \setlength{\rightmargin}{12mm}
     \setlength{\listparindent}{0mm}
     \setlength{\itemindent}{\listparindent}
     \setlength{\parsep}{0pt}
     \item\relax}}{\end{list}\vspace{.5\baselineskip}}
\begin{summary}
{\footnotesize {\bf Abstract.}
This paper deals with the quasilinear parabolic--elliptic Keller--Segel system with logistic source,
  \begin{align*}
    \begin{cases}
    u_t=\Delta (u+1)^m - \chi \nabla \cdot (u(u+1)^{\alpha - 1} \nabla v) 
          + \lambda(|x|) u - \mu(|x|) u^\kappa,
          & x\in\Omega,\ t>0,
  \\
    0=\Delta v - v + u,
       & x\in\Omega,\ t>0,
    \end{cases}
  \end{align*}
where $\Omega:=B_{R}(0)\subset\Rn\ (n\ge3)$ is a ball with some $R>0$;
$m>0$, $\chi>0$, $\alpha>0$ and $\kappa\ge1$; 
$\lambda$ and $\mu$ are spatially radial nonnegative functions.
About this problem, Winkler (Z. Angew.\ Math.\ Phys.; 2018; 69; Art.\ 69, 40) 
found the condition for $\kappa$ such that solutions blow up in finite time when $m=\alpha=1$. 
In the case that $m=1$ and $\alpha\in(0,1)$ as well as $\lambda$ and $\mu$ are constant, some conditions for $\alpha$ and $\kappa$ such that blow-up occurs 
were obtained in a previous paper (Math.\ Methods Appl.\ Sci.; 2020; 43; 7372--7396).
Moreover, in the case that $m\ge1$ and $\alpha=1$ Black, Fuest and Lankeit (arXiv:2005.12089[math.AP]) showed that there exists initial data such that solutions blow up in finite time under some conditions for $m$ and $\kappa$.
The purpose of the present paper is to give conditions for $m\ge1$, $\alpha>0$ and $\kappa\ge1$ such that solutions blow up in finite time.
} 
\end{summary}
\vspace{10pt}

\newpage
\section{Introduction}

%
%
The Keller--Segel system proposed by Keller and Segel \cite{K-S} in 1970 describes a part of the life cycle of cellular slime molds with chemotaxis.
After the pioneering work \cite{K-S}, a number of variations of the original Keller--Segel system are proposed and studied (see e.g., \cite{B-B-T-W,H-P,L-W}).

\medskip

In this present paper we consider finite-time blow-up in the following quasilinear parabolic--elliptic Keller--Segel system with logistic source:
\begin{equation}\label{PE}
  \begin{cases}
    u_t=\Delta (u+1)^m - \chi \nabla \cdot (u(u+1)^{\alpha - 1} \nabla v) 
          + \lambda(|x|) u - \mu(|x|) u^\kappa,
    & x\in\Omega,\ t>0, 
  \\
    0=\Delta v - v + u,
    & x\in\Omega,\ t>0,
  \\
    \nabla u \cdot \nu=\nabla v \cdot \nu=0,
    & x\in \pa\Omega,\ t>0,
  \\
    u(x,0)=u_0(x),
    & x\in\Omega,
  \end{cases}
\end{equation}
where $\Omega=B_{R}(0)\subset\Rn\ (n\in\mathbb{N})$ be a ball with some $R>0$;  
$m>0$, 
$\chi>0$, $\alpha>0$, 
$\kappa\ge1$ are constants;
$\lambda$ and $\mu$ are spatially radial nonnegative functions and $\mu(r)\le\mu_1r^q$ for all $r\in[0,R]$ with some $\mu_1>0$ and $q\ge0$;
$\nu$ is the outward normal vector to $\pa\Omega$;  
%
\[
  u_0 
  \in C^0(\overline{\Omega})\ \mbox{is radially symmetric and nonnegative.} 
\]
%
The unknown functions $u=u(x,t)$ and $v=v(x,t)$ represent the density of cells and the concentration of the chemoattractant at $x\in\Omega$ and $t\ge0$, respectively.

\medskip

%
%

From a mathematical point of view, it is a meaningful question whether solutions blow up or remain bounded.
Before we introduce previous works about the system \eqref{PE}, 
we recall some known results related to this problem. 
In the quasilinear Keller--Segel system 
  \begin{align}\label{quasi}
    \begin{cases}
    u_t=\Delta (u+1)^m - \chi \nabla \cdot (u(u+1)^{\alpha - 1} \nabla v),
    & x\in\Omega,\ t>0, 
  \\
    \tau v_t=\Delta v - v + u,
    & x\in\Omega,\ t>0,
    \end{cases}
  \end{align}
where $m$, $\alpha\in\Rone$, $\chi>0$ and $\tau\in\{0,1\}$, 
it is known that the relation between $m$ and $\alpha$ determines the properties of solutions to system \eqref{quasi}.
In the parabolic--parabolic setting $(\tau=1)$ Tao and Winkler \cite{T-W} proved that global bounded solutions exist under the conditions that $m-\alpha>\frac{n-2}{n}$ and that $\Omega$ is a bounded convex domain with smooth boundary, 
and the convexity of $\Omega$ was completely removed in \cite{I-S-Y};
whereas, Winkler \cite{W_2010_quasilinear} showed that there exist initial data such that the solution blows up in either finite or infinite time if $m-\alpha<\frac{n-2}{n}$;
moreover, Cie\'{s}lak and Stinner \cite{C-S-2012,C-S-2015} obtained existence of finite-time blow-up solutions in the case that $m-\alpha<\frac{n-2}{n}$ and either $m\ge1$ or $\alpha\ge1$; 
Winkler \cite{W-2019} established infinite-time blow-up when $m-\alpha<\frac{n-2}{n}$ and $\alpha\le0$.
In the parabolic--elliptic setting $(\tau=0)$ Lankeit \cite{L-2020} asserted that if $m-\alpha>\frac{n-2}{n}$, then solutions are global and bounded and that if $m-\alpha<\frac{n-2}{n}$, then there exist initial data such that the solution blows up in either finite or infinite time; in particular, if $m-\alpha<\frac{n-2}{n}$ and $\alpha\le0$, then the solution blows up in infinite time;
furthermore, in the case that the second equation of the system \eqref{quasi} is $0=\Delta v -M+u$ instead of $\tau v_t=\Delta v - v + u$, where $M:=\frac{1}{|\Omega|}\int_{\Omega}u_0$, Cie\'{s}lak and Winkler \cite{C-W} showed global existence and boundedness in the case that $\alpha=2$ and $2<m+\frac{2}{n}$ and existence of finite-time blow-up solutions in the case that $\alpha=2$ and $2>m+\frac{2}{n}$; 
Winkler and Djie \cite{W-D} proved that there exist initial data such that the solution blows up in finite time when $m-\alpha<\frac{n-2}{n}$ and $\alpha>0$ and that global bounded solutions exist when $m-\alpha>\frac{n-2}{n}$.

\medskip

In the study of the quasilinear Keller--Segel system with logistic source 
  \begin{align}\label{quasi-logi}
    \begin{cases}
    u_t=\Delta (u+1)^m - \chi \nabla \cdot (u(u+1)^{\alpha - 1} \nabla v)
         +\lambda u-\mu u^\kappa,
    & x\in\Omega,\ t>0, 
  \\
    \tau v_t=\Delta v - v + u,
    & x\in\Omega,\ t>0,
    \end{cases}
  \end{align}
where $m$, $\alpha\in\Rone$, $\kappa\ge1$, $\chi>0$, $\lambda\ge0$, $\mu>0$ and $\tau\in\{0,1\}$, 
the logistic term $\lambda u-\mu u^\kappa$ suppresses blow-up phenomena in the case that $\kappa\ge2$.
In the system \eqref{quasi-logi} with $m=1$, $\alpha=1$ and $\tau=1$,
if $\kappa=2$ and $\mu>\mu_0$ for some $\mu_0>0$, then Winkler \cite{W-2010} derived that all solutions are global and bounded; 
in two-dimensional setting, 
global existence and boundedness of solutions were 
proved for all $\mu>0$ in \cite{O-T-Y-M};
in the parabolic--elliptic setting $(\tau=0)$,
Tello and Winkler \cite{Tello-W-2007} showed that global bounded solutions exist when $\kappa=2$ and $\mu>\max\left\{0,\frac{n-2}{n}\chi\right\}$ and when $\kappa>2$ and $\mu>0$.
In the system \eqref{quasi-logi} with $m$, $\alpha\in\Rone$ and $\tau=1$,
Zheng \cite{Z-2017} asserted that if $\kappa=2$, $\lambda=\mu=1$ and $0<1-m+\alpha<\frac{4}{n+4}$, then solutions are global in time and bounded;
in the parabolic--elliptic setting $(\tau=0)$, 
global existence and boundedness of solutions were 
established under the conditions that $m\ge1$, $\alpha>0$, $\kappa>1$ and $\alpha+1<\max\left\{\kappa,m+\frac{2}{n}\right\}$ and that $m\ge1$, $\alpha>0$, $\kappa>1$ and $\alpha+1=\kappa$ and $\mu>\mu_0$ for some $\mu_0>0$ in \cite{Z-2015}.

\medskip

From the above results about the system \eqref{quasi-logi}, one might imply that the logistic term $\lambda u-\mu u^\kappa$ suppresses blow-up. 
However, on the contrary, Winkler \cite{W-2018} found the condition for $\kappa>1$ such that there exists an initial data such that the corresponding solution blows up in finite time in the sytem \eqref{quasi-logi} with $m=1$, $\alpha=1$ and $\tau=0$. 
For the detail, an initial data such that finite-time blow-up occurs can be obtained under the condition that
  \begin{align*}
    \kappa<
      \begin{cases}
        \frac{7}{6}\quad &\mbox{if}\ n\in\{3,4\},
      \\
        1+\frac{1}{2(n-1)}\quad &\mbox{if}\ n\ge5.
      \end{cases}
  \end{align*}
Moreover, in the case that $m=1$ and $\alpha=1$ as well as the second equation of \eqref{quasi-logi} is $0=\Delta v -\overline{M}(t)+u$, where $\overline{M}(t):=\frac{1}{|\Omega|}\int_{\Omega}u$, Winkler \cite{W-2011} has already obtained the condition that $\kappa<\frac{3}{2}+\frac{1}{2n-2}$ when $n\ge5$;
recently, similar blow-up results were obtained in the case that $m\ge1$ and $\alpha=1$ in \cite{B-F-L};
furthermore, Fuest \cite{F_2021_optimal} asserted that the exponent $\kappa=2$ is critical in the four and higher $(n\ge5)$ dimensional setting when $m=1$ and $\alpha=1$.
In the system \eqref{quasi-logi} with $m=1$ and $\tau=0$, some conditions for $\kappa>1$ and $\alpha\in(0,1)$ such that there exist initial data which lead to blow-up were found in \cite{Tanaka-Y_2020}.
On the other hand, in the system \eqref{PE} with $\alpha=1$ and $\tau=0$, Black, Fuest and Lankeit \cite{B-F-L} constructed initial data such that the solution blows up under some conditions for $\kappa\ge1$ and $m\in\left[1,\frac{2n-2}{n}\right)$.

\medskip

In summary, these results \cite{B-F-L,F_2021_optimal,Tanaka-Y_2020,W-2011,W-2018} imply that blow-up occurs when the exponent $\kappa$ of logistic source is small.
In particular, in the system with nonlinear diffusion \cite{B-F-L}, that is, in the system \eqref{PE} with $\alpha=1$ the conditions such that the solutions blow-up was found. 
However, we have not obtained conditions which lead to blow-up when $m>0$ and $\alpha>0$ in the system \eqref{PE}.
The purpose of this paper is to give conditions for $m$, $\alpha$ and $\kappa$ such that the solutions of \eqref{PE} blow up. 

\medskip

%
%
Let $p\ge n$. 
In order to state the main theorem
we define the conditions (A1)--(A4), (B1)--(B3), (C1)--(C3) and (D1)--(D2) as follows:
\begin{itemize}
  \item In the case $n=3$,
    \begin{align*}
      \tag{A1}  &1-\frac{1}{p}<\alpha<1+\frac{3}{2p}
                      ,\quad
                      0<m<1+\frac{1}{p},
      \\
      \tag{A2}  &1+\frac{3}{2p}\le\alpha<1+\frac{2}{p},\quad
                      0<m<\frac{2}{p}
                      ,\quad
                      2\alpha-m>2+\frac{2}{p},
      \\
      \tag{A3}  &1-\frac{1}{p}<\alpha<1+\frac{2}{p},\quad
                      \frac{1}{p}\le m<\frac{2}{p}
                      ,\quad
                      2\alpha-m\le2+\frac{2}{p},
                    \\
                    &\mbox{or}\quad
                      1-\frac{1}{p}<\alpha<1,\quad
                      \frac{2}{p}\le m<\frac{3}{p}
                      ,\quad
                      m+\alpha<1+\frac{2}{p},
      \\
      \tag{A4}  &1-\frac{1}{p}<\alpha<1+\frac{2}{p},\quad
                      \frac{2}{p}\le m<1+\frac{1}{p},\quad
                      m+\alpha\ge1+\frac{2}{p}
                      ,\quad
                      m-\alpha<\frac{1}{p}.
    \end{align*}
  \item In the case $n=4$,
    \begin{align*}
      \tag{B1}  &1-\frac{2}{p}<\alpha<1+\frac{2}{p}
                      ,\quad
                      0<m<\frac{2}{p},
      \\
      \tag{B2}  &1-\frac{2}{p}<\alpha<1,\quad
                      \frac{2}{p}\le m<\frac{4}{p}
                      ,\quad
                      m+\alpha<1+\frac{2}{p},
      \\
      \tag{B3}  &1-\frac{2}{p}<\alpha<1+\frac{2}{p},\quad
                      \frac{2}{p}\le m<1+\frac{2}{p},\quad
                      m+\alpha\ge1+\frac{2}{p}
                      ,\quad
                      m-\alpha<\frac{2}{p}.
    \end{align*}
  \item In the case $n=5$,
    \begin{align*}
      \tag{C1}  &1-\frac{2}{p}<\alpha\le1-\frac{1}{p}
                      ,\quad
                      0<m<\frac{3}{p},
                    \\
                    &\mbox{or}\quad 
                      1-\frac{1}{p}<\alpha<1+\frac{2}{p},\quad
                      0<m<1+\frac{1}{2p}
                      ,\quad
                      2m-\alpha<1+\frac{1}{p},
      \\
      \tag{C2}  &1-\frac{2}{p}<\alpha<1-\frac{1}{p},\quad
                      \frac{3}{p}\le m<\frac{4}{p}
                      ,\quad
                      m+\alpha<1+\frac{2}{p},
      \\ 
      \tag{C3}  &1-\frac{2}{p}<\alpha\le1-\frac{1}{p},\quad
                      \frac{3}{p}\le m<1
                      ,\quad
                      m+\alpha\ge1+\frac{2}{p},
                    \\
                    &\mbox{or}\quad
                      1-\frac{2}{p}<\alpha<1,\quad
                      1\le m<1+\frac{1}{2p}
                      ,\quad
                      2m-\alpha\ge1+\frac{1}{p},
                    \\
                    &\mbox{or}\quad
                      1-\frac{2}{p}<\alpha<1+\frac{2}{p},\quad
                      1+\frac{1}{2p}\le m<1+\frac{3}{p}
                      ,\quad
                      m-\alpha<\frac{3}{p}.
    \end{align*}
  \item In the case $n\ge6$,
    \begin{align*}
      \tag{D1}  &1-\frac{2}{p}<\alpha<1+\frac{2}{p},\quad
                      0<m<1+\frac{n-4}{2p}
                      ,\quad
                      2m-\alpha<1+\frac{n-4}{p},
      \\
      \tag{D2}  &1-\frac{2}{p}<\alpha<1+\frac{2}{p},\quad
                      1+\frac{n-6}{2p}\le m<1+\frac{n-4}{2p}
                      ,\quad
                      2m-\alpha\ge1+\frac{n-4}{p},
                    \\
                    &\mbox{or}\quad
                      1-\frac{2}{p}<\alpha<1+\frac{2}{p},\quad
                      1+\frac{n-4}{2p}\le m<1+\frac{n-2}{p}
                      ,\quad
                      m-\alpha<\frac{n-2}{p}.
    \end{align*}
\end{itemize}
Moreover, we assume that $\kappa\ge1$ fulfills the following conditions:
  \begin{align}
  \tag{{\rm {\bf I}}}
    &\kappa<1+\frac{3}{p}+\frac{q}{p}-(\alpha-1)
    &\quad &\mbox{if (A2) holds},
  \\
  \tag{{\rm {\bf I\hspace{-.1em}I}}}
    &\kappa<1+\frac{n}{2p}+\frac{q}{p}-\frac{(1-\alpha)_{+}}{2}
    &\quad &\mbox{if (A1), (B1), (C1) or (D1) hold},
  \\
  \tag{{\rm {\bf I\hspace{-.1em}I\hspace{-.1em}I}}}
    &\kappa<1+\frac{n-1}{p}+\frac{q}{p}-\frac{m}{2}-\frac{(1-\alpha)_{+}}{2}
    &\quad &\mbox{if (A3), (B2) or (C2) hold},
  \\ 
  \tag{{\rm {\bf I\hspace{-.1em}V}}}
    &\kappa<1+\frac{n-2}{p}+\frac{q}{p}-(m-1)_{+}-(1-\alpha)_{+}
    &\quad &\mbox{if (A4), (B3), (C3) or (D2) hold},
  \end{align}
where $(m-1)_{+}:=\max\{0,m-1\}$ and $(1-\alpha)_{+}:=\max\{0,1-\alpha\}$.
Here, the regions derived from the conditions (A1)--(D2) are arranged as follows.

\begin{figure}[h]
\begin{minipage}{0.49\columnwidth}
\hspace*{3cm}
\scalebox{0.83}{\includegraphics{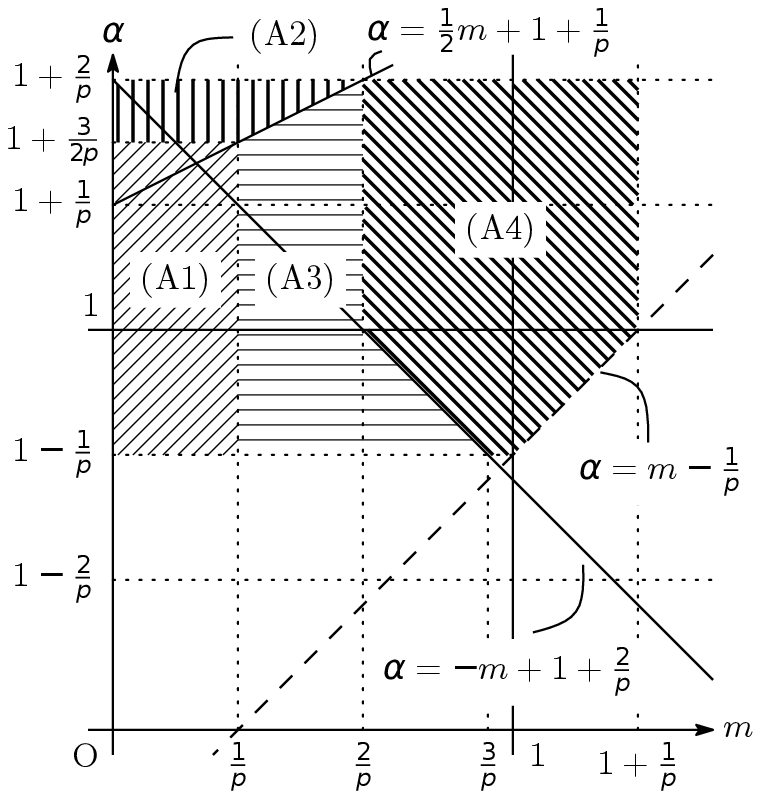}}
\caption{$n=3$}
\end{minipage}
\begin{minipage}{0.49\columnwidth}
\hspace*{3.09cm}
\scalebox{0.83}{\includegraphics{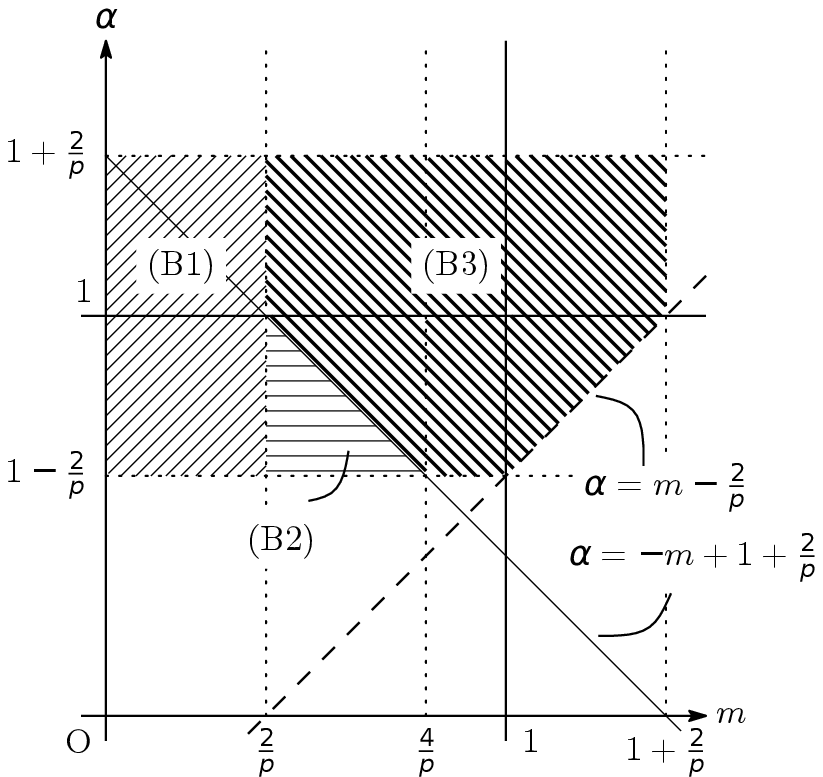}}
\caption{$n=4$}
\end{minipage}

\bigskip

\begin{minipage}{0.49\columnwidth}
\hspace*{0.66cm}
\scalebox{0.83}{\includegraphics{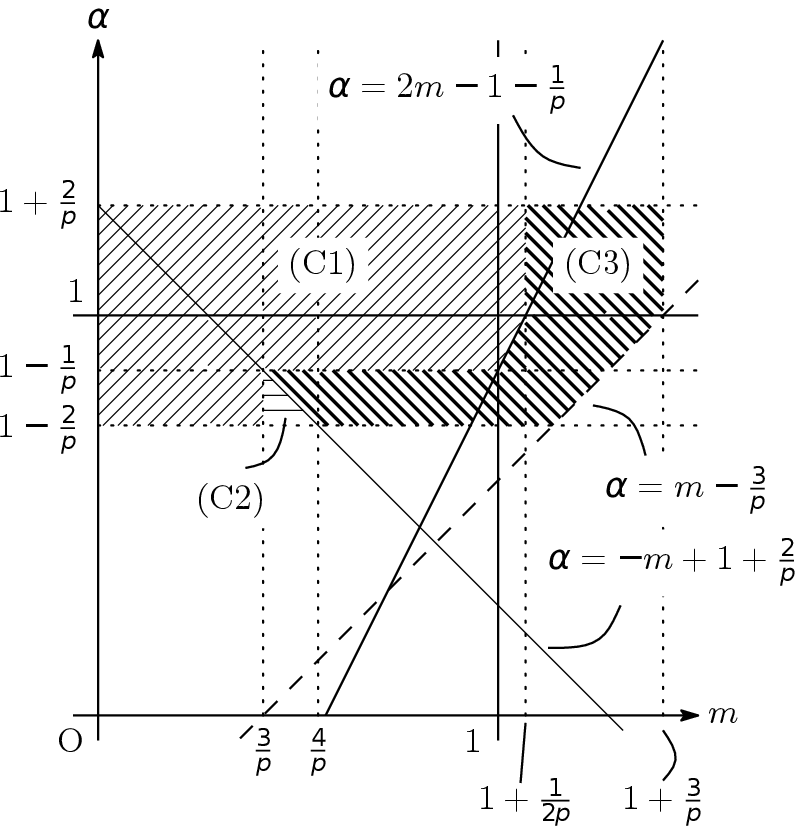}}
\caption{$n=5$}
\end{minipage}
\begin{minipage}{0.49\columnwidth}
\hspace*{0.64cm}
\scalebox{0.83}{\includegraphics{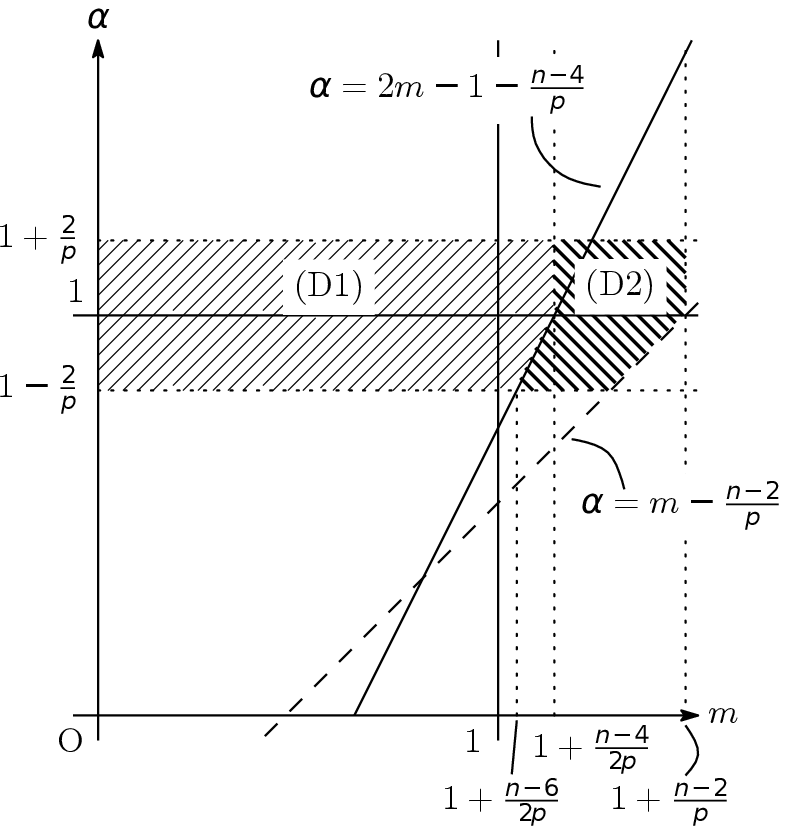}}
\caption{$n=6$}
\end{minipage}
\end{figure}%
%
%
%
%
%
%

Now we state the main theorems.
The first result is concerned with blow-up when we assume an upper bound of solutions.
%
%
%
\begin{thm}\label{mainthm1}
Let $\Omega=B_R(0)\subset\Rn\ (n\ge3)$ with $R>0$ 
and let $m>0$, $\alpha>0$, $\chi>0$, $\kappa\ge1$, $\mu_1>0$, $p\ge n$, $q\ge0$,
$M_0>0$, $M_1\in(0,M_0)$, $K>0$ and $T>0$.
Suppose that $\lambda$ and $\mu$ satisfy that
%
  \begin{align}\label{lm}
    0\le\lambda, \mu\in C^0([0,R]) 
  \end{align}
and
  \begin{align}\label{mu}
    \mu(r)\le\mu_1r^q\quad\mbox{for all}\ r\in[0,R]
  \end{align}
%
and assume that $\kappa$ fulfills {\rm ({\bf I})}--{\rm ({\bf I\hspace{-.1em}V})}.
Then one can find $r_1\in(0,R)$ with the following property\/{\rm :} 
If 
%
\begin{align*}
  \begin{cases}
    u\in C^0(\overline{\Omega}\times[0,\tmax))\cap C^{2,1}(\overline{\Omega}\times(0,\tmax)),\\
    v\in \bigcap_{\vartheta>n}C^0([0,\tmax);W^{1,\vartheta}(\Omega))\cap C^{2,1}(\overline{\Omega}\times(0,\tmax)),
  \end{cases}
\end{align*}
is a classical solution to \eqref{PE} for some $T^*\in(0,\infty]$ with 
  \begin{align}\label{initial}
      u_0 
      \in C^0(\overline{\Omega})\ 
      \mbox{being radially symmetric and nonnegative} 
  \end{align}
and
  \[
      \int_\Omega u_0=M_0
      \quad\mbox{but}\quad
      \int_{B_{r_1}(0)} u_0 \ge M_1
  \]
as well as
%
\begin{align}\label{power}
  \sup_{t\in(0,{\rm min}\{T,T^*\})}u(x,t)\le K|x|^{-p}\quad\mbox{for all}\ x\in\Omega,
\end{align}
%
then $(u,v)$ blows up at $t=T^*<\infty$ in the sense that
  \begin{align}\label{blowup}
      \limsup_{t\nearrow T^*} \|u(\cdot,t)\|_{L^\infty(\Omega)}=\infty.
  \end{align}
\end{thm}
%
%
%
%
\begin{remark}
As to the conditions (A1)--(D2) and {\rm ({\bf I})}--{\rm ({\bf I\hspace{-.1em}V})}, if $\alpha=1$, then we can obtain the conditions such that
  \begin{align*}
    1&\le\kappa<1+\frac{q}{p}+\min\left\{\frac{n}{2p},\frac{n-2}{p}-(m-1)_{+}\right\}
    \quad&\mbox{if}\ &m\in\left[\frac{2}{p},1+\frac{n-2}{p}\right)
  \\
    \mbox{or}\quad
    1&\le\kappa<1+\frac{q}{p}+\min\left\{\frac{n}{2p},\frac{n-1}{p}-\frac{m}{2}\right\}
    \quad&\mbox{if}\ &m\in\left(0,\frac{2}{p}\right).
  \end{align*}
The above conditions for $m$ and $\kappa$ connect with the conditions in \cite[Theorem 1.1]{B-F-L}.
Thus, Theorem \ref{mainthm1} is a generalization of the previous work \cite[Theorem 1.1]{B-F-L}.
\end{remark}

By an argument similar to that in the proof of \cite[Lemma 5.2]{B-F-L} we can find initial data such that the corresponding solution satisfies \eqref{power}.
Therefore, in view of Theorem \ref{mainthm1} we can show that there exist initial data such that the solution blows up in finite time.
Before we introduce this result, we define the conditions (E1), (F1) and (F2) as follows:
  \begin{itemize}
    \item In the case $n\in\{3,4\}$,
      \begin{align}
      \tag{E1}
        m\ge1,\ 
        \alpha<\frac{2}{n+1}m+\frac{n^2-n+2}{n(n+1)},
      \ 
        \alpha<-\frac{1}{n-2}m+\frac{n^2-2}{n(n-2)},\ 
        m-\alpha<\frac{n-2}{n}.
      \end{align}
    \item In the case $n\ge5$,
      \begin{align}
      \tag{F1}
        &m\ge1,\quad 
        -\frac{2}{n-3}m+\frac{n^2-n-2}{n(n-3)}<\alpha<\frac{2}{n+1}m+\frac{n^2-n+2}{n(n+1)},
      \\ \notag
        &\alpha<-\frac{n+2}{n-4}m+\frac{2n^2-n-4}{n(n-4)},\quad 
        \alpha\le\frac{n+2}{3}m-\frac{n^2-4}{3n},
      \\
      \tag{F2}
        &m\ge1,\quad 
        -\frac{2}{n-3}m+\frac{n^2-n-2}{n(n-3)}<\alpha<\frac{2}{n+1}m+\frac{n^2-n+2}{n(n+1)},
      \\ \notag
        &-\frac{n+2}{n-4}m+\frac{2n^2-n-4}{n(n-4)}\le\alpha
          <-\frac{1}{n-2}m+\frac{n^2-2}{n(n-2)},\quad
        m-\alpha<\frac{n-2}{n}.
      \end{align}
  \end{itemize}
The regions of (E1), (F1) and (F2) are described as follows:

\begin{figure}[h]
\begin{minipage}{0.49\columnwidth}
\hspace*{3.1cm}
\scalebox{0.87}{\includegraphics{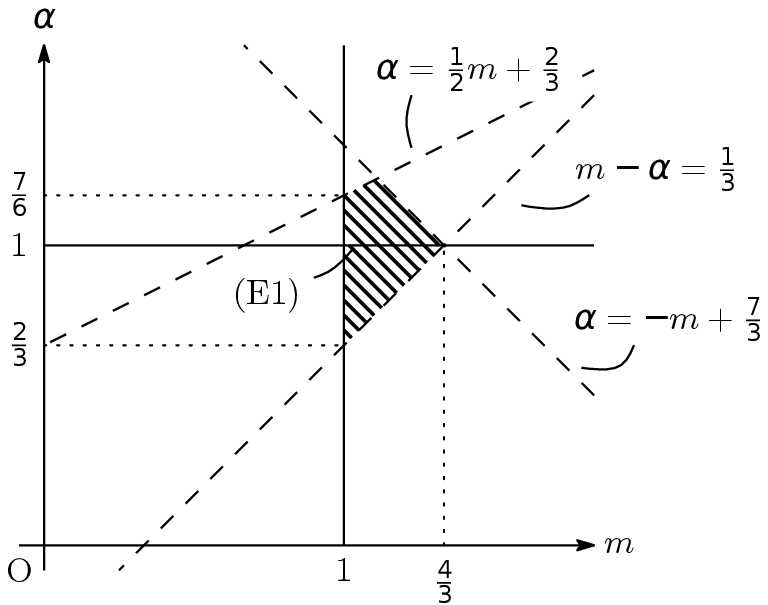}}
\caption{$n=3$}
\end{minipage}
\hspace*{-1cm}
\begin{minipage}{0.49\columnwidth}
\hspace*{3.09cm}
\scalebox{0.87}{\includegraphics{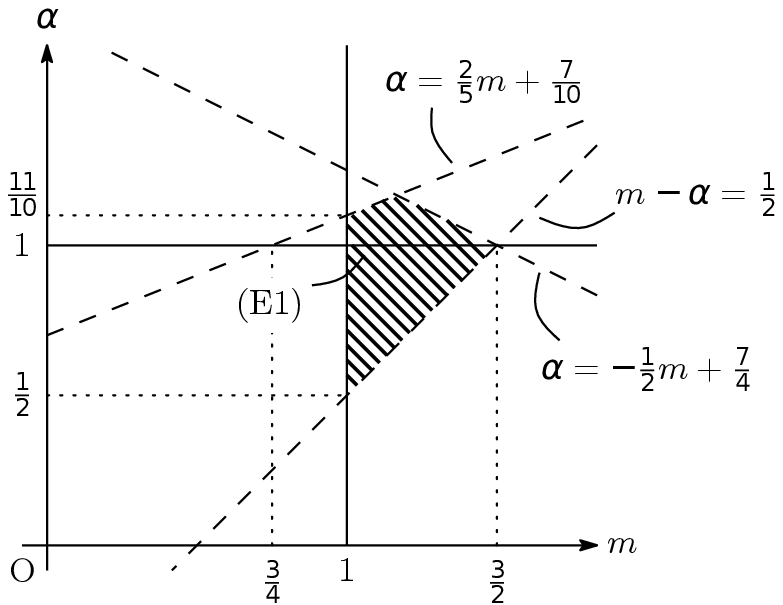}}
\caption{$n=4$}
\end{minipage}

\bigskip

\begin{minipage}{0.49\columnwidth}
\hspace*{3.1cm}
\scalebox{0.87}{\includegraphics{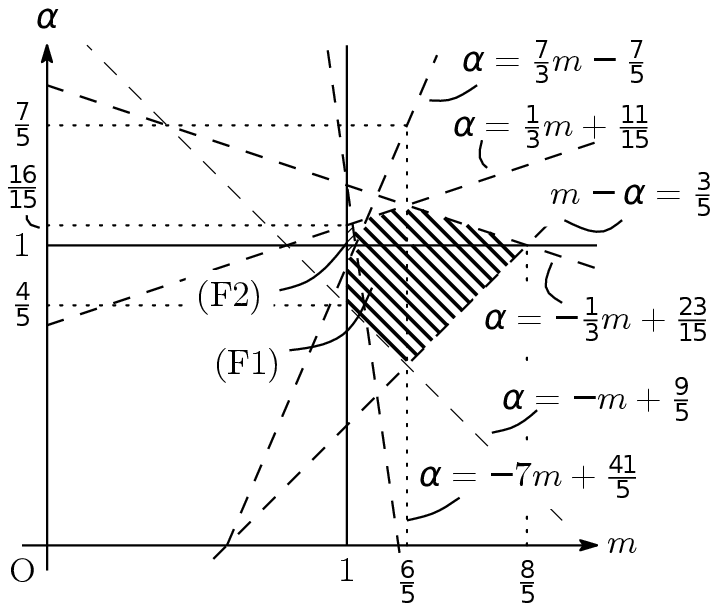}}
\caption{$n=5$}
\end{minipage}
\hspace*{-1cm}
\begin{minipage}{0.49\columnwidth}
\hspace*{3cm}
\scalebox{0.87}{\includegraphics{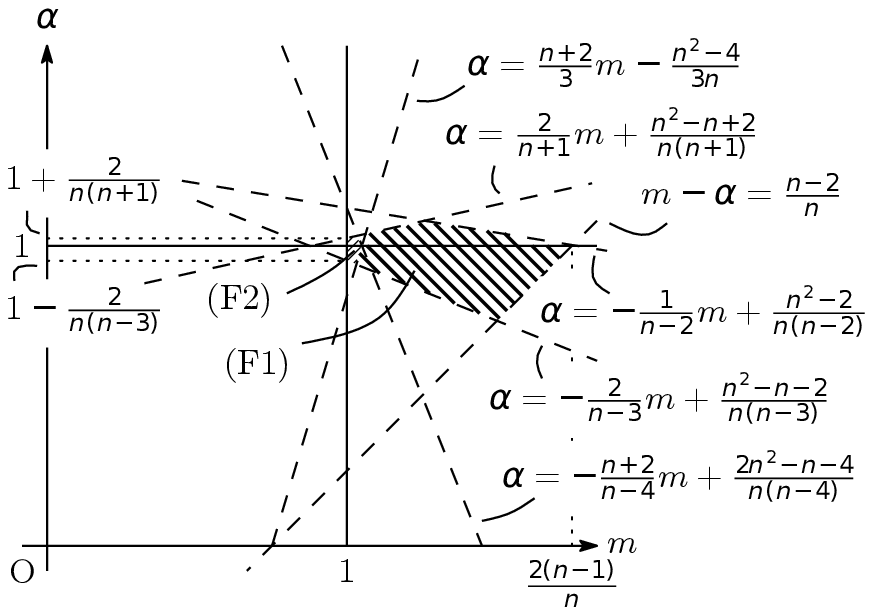}}
\caption{$n=6$}
\end{minipage}
\end{figure}%
%
%
%

Aided by Theorem \ref{mainthm1}, we obtain an initial data such that the solution blows up.

%
%
\begin{thm}\label{mainthm2}
Let $\Omega=B_R(0)\subset\Rn\ (n\ge3)$ with $R>0$ 
and let $m>0$, $\alpha>0$, $\chi>0$, $\kappa\ge1$, $\mu_1>0$, $q\ge0$, $M_0>0$, $M_1\in(0,M_0)$ and $L>0$. 
Suppose that $\lambda$ and $\mu$ satisfy \eqref{lm} and \eqref{mu}.
Moreover, assume that $m$, $\alpha$ and  $\kappa$ fulfill the following conditions\/{\rm :}
  \begin{itemize}
    \item[{\rm (i)}] If {\rm (E1)} holds, then 
      \[
        \kappa<1+\frac{(n-2)[(m-\alpha)n+1]}{n(n-1)}+\frac{q[(m-\alpha)n+1]}{n(n-1)}
                    -(m-1)-(1-\alpha)_{+}.
      \]
    \item[{\rm (ii)}] If {\rm (F1)} holds, then 
      \[
        \kappa<1+\frac{(n-2)[(m-\alpha)n+1]}{n(n-1)}+\frac{q[(m-\alpha)n+1]}{n(n-1)}
                    -(m-1)-(1-\alpha)_{+}.
      \]
    \item[{\rm (iii)}] If {\rm (F2)} holds, then
      \[
        \kappa<1+\frac{(m-\alpha)n+1}{2(n-1)}+\frac{q[(m-\alpha)n+1]}{n(n-1)}
                    -\frac{(1-\alpha)_{+}}{2}.
      \]
  \end{itemize}
Then one can find $\varepsilon_0>0$ and $r_1\in(0,R)$ with the following property\/{\rm :} 
If $u_0$ with \eqref{initial} satisfies $\int_\Omega u_0=M_0$ and $\int_{B_{r_1}(0)} u_0 \ge M_1$
as well as $u_0(x)\le L|x|^{-p}$ for all $x\in\Omega$, where $p:=\frac{n(n-1)}{(m-\alpha)n+1}+\varepsilon_0$, then the solution $(u,v)$ to \eqref{PE} fulfills \eqref{blowup} for some $T^*<\infty$. 
\end{thm}
%
%
%
%
\begin{remark}
If $\alpha=1$, then we have from the conditions (E1)--(F2) and (i)--(iii) that 
  \[
    \kappa<1+\frac{q[(m-1)n+1]}{n(n-1)}
                     +\min\left\{\frac{(m-1)n+1}{2(n-1)},\frac{n-2-(m-1)n}{n(n-1)}\right\}
    \quad \mbox{if}\ 
m\in\left[1,\frac{2n-2}{n}\right)
  \]
which is the condition in \cite[Theorem 1.2]{B-F-L}.
Thus, Theorem \ref{mainthm2} is a generalization of the previous work \cite[Theorem 1.2]{B-F-L}.
On the other hand, if $m=1$, $\alpha<1$ and $q=0$, then we can obtain that
  \begin{align}
  \label{a1}
    &\mbox{if}\ n\in\{3,4\}, &\ &\frac{2}{n}<\alpha<1
                                    \quad \mbox{and}
                                    \quad \kappa<1+\frac{(n-2)-(1-\alpha)n}{n(n-1)},
  \\
  \label{a2}
    &\mbox{if}\ n=5, &\ &
      \begin{cases}
        \dfrac{4}{5}<\alpha\le\dfrac{14}{15}
        \quad\mbox{and}\quad
        \kappa<1+\dfrac{(n-2)-(1-\alpha)n}{n(n-1)},
      \\
        \dfrac{14}{15}<\alpha<1
        \quad\mbox{and}\quad
        \kappa<1+\dfrac{2-\alpha}{2(n-1)},
      \end{cases} 
  \\
  \label{a3}
    &\mbox{if}\ n\ge6, &\ &1-\frac{2}{n(n-3)}<\alpha<1
                              \quad \mbox{and} 
                              \quad \kappa<1+\frac{2-\alpha}{2(n-1)}.
  \end{align}
The conditions \eqref{a1}--\eqref{a3} improve lower bounds for $\alpha$ and upper bounds for $\kappa$ in \cite{Tanaka-Y_2020}.
Moreover, while blow-up result was only obtained in the case $\alpha<1$ in \cite{Tanaka-Y_2020}, we can see that the result is extend to the condition $\alpha\ge1$.
\end{remark}
%
%
%
%
\begin{remark}
In the conditions (E1)--(F2) there are some restrictions in addition to the condition $m-\alpha<\frac{n-2}{n}$.
On the one hand, in the system \eqref{quasi} it is known that there exist many results about blow-up under the condition $m-\alpha<\frac{n-2}{n}$ without these restrictions
(see \cite{C-S-2012,C-S-2015,C-W,L-2020,W_2010_quasilinear,W-2019,W-D}).
Thus, Theorem \ref{mainthm2} may hold under the condition $m-\alpha<\frac{n-2}{n}$ even without these restrictions.
\end{remark}

The proofs of Theorems \ref{mainthm1} and \ref{mainthm2} are based on those of \cite{B-F-L}.
We first introduce the mass accumulation functions $w=w(s,t)$ and $z=z(s,t)$ given by
  \[
    w(s,t):=\int^{s^\frac{1}{n}}_{0}\rho^{n-1}u(\rho,t)\,d\rho,
    \quad
    z(s,t):=\int^{s^\frac{1}{n}}_{0}\rho^{n-1}v(\rho,t)\,d\rho,
  \]
where $s:=r^n$ for $r\in[0,R]$. 
The system \eqref{PE} is transformed to the parabolic equation
  \begin{align}\label{paraequ}
     w_t&= n^2ms^{2-\frac{2}{n}}(nw_s+1)^{m-1}w_{ss} +\chi nw_s(nw_s+1)^{\alpha-1}(w-z)
            \\ \notag
            &\quad\,
                +n\int^{s}_{0}\lambda(\sigma^\frac{1}{n})w_s(\sigma,t)\,d\sigma
                -n^{\kappa-1}\int^{s}_{0}\mu(\sigma^\frac{1}{n})w^{\kappa}(\sigma,t)\,d\sigma.
  \end{align}
Next, from \eqref{paraequ} and the moment-type functional
  \[
    \phi(s_0,t):=\int^{s_0}_{0}s^{-\gamma}(s_0-s)w(s,t)\,ds
  \]
with some $s_0\in(0,R^n)$ and $\gamma\in(0,1)$ we will show that the functional $\phi$ is a supersolution 
of the ordinary differential equation 
$\phi'=c_1\phi^2-c_2$ with some $c_1>0$ and $c_2>0$.
Here, as to the factor $(nw_s+1)^{m-1}$ of the first term on the right-hand side of \eqref{paraequ} we can apply the same estimates as in \cite{B-F-L}. 
However, in order to derive a differential inequality for $\phi$ we have to estimate the factor $(nw_s+1)^{\alpha-1}$ of the second term on the right-hand side of \eqref{paraequ}.
Therefore, in the case $\alpha<1$ we use the estimates $(nw_s+1)^{\alpha-1}\le1$ and 
$(nw_s+1)^{\alpha-1}\ge (Cs^{-\frac{p}{n}}+1)^{\alpha-1}$ as in \cite{Tanaka-Y_2020} and 
in the case $\alpha\ge1$ we establish the estimates $(nw_s+1)^{\alpha-1}\le(Cs^{-\frac{p}{n}}+1)^{\alpha-1}$ and $(nw_s+1)^{\alpha-1}\ge1$ on a case by case basis.
Moreover, by introducing the conditions for $\gamma\in(0,1)$ (see \eqref{lem4.1} and \eqref{lem4.2}), we can obtain a differential inequality for $\phi$.
As to the proof of Theorem \ref{mainthm2}, we can obtain initial data such that the solution fulfills \eqref{power} by the recent study of blow-up profiles in \cite{F-2020_profiles}.

\medskip

This paper is organized as follows.
In Section $2$ we recall local existence of classical solutions in \eqref{PE}.
In Section $3$ we estimate a differential inequality for $\phi$ in order to construct a subsolution.
In Section $4$ we prove existence of $\gamma\in(0,1)$ which satisfies conditions to derive a differential inequality for $\phi$ and obtain a super-quadratic nonlinear differential inequality.
Finally, the proofs of the main theorems are given in Section $5$.

\section{Local existence}
We first introduce a result on local existence of classical solutions to \eqref{PE}.
We provide only the statement of the lemma since the proof is based on a standard fixed point argument (see \cite{C-W,Tello-W-2007}).
%
%
\begin{lem}\label{solution}
Let $n\ge1$, $R>0$, $m>0$, $\alpha>0$, $\chi>0$, $\kappa\ge1$ and $M_0>0$,
and assume that $\lambda$ and $\mu$ comply with \eqref{lm} and \eqref{mu}.
If $u_0$ satisfies \eqref{initial} and $\int_\Omega u_0=M_0$, then there exist $\tmax\in(0,\infty]$ and an exactly pair $(u,v)$ of radially symmetric nonnegative functions
\begin{align*}
  \begin{cases}
    u\in C^0(\overline{\Omega}\times[0,\tmax))\cap C^{2,1}(\overline{\Omega}\times(0,\tmax)),\\
    v\in \bigcap_{\vartheta>n}C^0([0,\tmax);W^{1,\vartheta}(\Omega))
\cap C^{2,1}(\overline{\Omega}\times(0,\tmax)),
  \end{cases}
\end{align*}
which solves \eqref{PE} classically.
Moreover,
    \begin{align*}
      \mbox{if}\ \tmax<\infty,\ \mbox{then}\ 
      \limsup_{t\nearrow\tmax}\|u(\cdot,t)\|_{L^{\infty}(\Omega)}=\infty.
    \end{align*}
\end{lem}
\section{Differential inequality for a moment type functional}
In the following let $\Omega=B_{R}(0)\subset\Rn\ (n\ge3)$ be a ball with some $R>0$ and let $(u,v)$ be the radially symmetric solution of \eqref{PE} on $[0,\tmax)$ as in Lemma \ref{solution}.
By introducing $r:=|x|$, we regard $u(x,t)$ and $v(x,t)$ as $u(r,t)$ and $v(r,t)$, respectively.
Based on \cite{J-L}, we define the mass accumulation functions $w$ and $z$ as
%
  \[
    w(s,t):=\int^{s^\frac{1}{n}}_{0}\rho^{n-1}u(\rho,t)\,d\rho,
    \quad
%
    z(s,t):=\int^{s^\frac{1}{n}}_{0}\rho^{n-1}v(\rho,t)\,d\rho
    \quad
  \]
%
for all $s\in[0,R^n]$ and $t\in[0,\tmax)$.
Moreover, given $s_0\in(0,R^n)$ and $\gamma\in(0,1)$, we set
%
\begin{align}\label{phi}
  \phi(s_0,t):=\int^{s_0}_{0}s^{-\gamma}(s_0-s)w(s,t)\,ds\quad\mbox{for all}\ t\in[0,\tmax), 
\end{align}
%
which is introduced in \cite{B-F-L,W-2018}, and
%
%
\[
  \psi_\alpha(s_0,t):=\int^{s_0}_{0}s^{-\gamma+\frac{p}{n}(1-\alpha)_{+}}(s_0-s)w(s,t)w_s(s,t)\,ds
  \quad\mbox{for all}\ t\in[0,\tmax),
\]
%
where $p\ge n$, $\alpha>0$ and $(1-\alpha)_{+}:={\rm max}\{0,1-\alpha\}$.
Now we recall the following properties for the functions $w$ and $z$ (see \cite[Lemma 3.1]{B-F-L}).
%
%
%
\begin{lem}
We have
  \begin{align*}
    w\in C^{1,0}([0,R^n]\times[0,\tmax))
              \cap C^{2,1}([0,R^n]\times(0,\tmax))
              \cap C^{3,0}((0,R^n]\times(0,\tmax))&,\\\nonumber
    z\in C^{1,0}([0,R^n]\times[0,\tmax))
              \cap C^{2,0}([0,R^n]\times(0,\tmax))
              \cap C^{3,0}((0,R^n]\times(0,\tmax))&
  \end{align*}
and
  \begin{align*}
      w_s(s,t)&=\frac{1}{n}u(s^\frac{1}{n},t),
      \quad
      w_{ss}(s,t)=\frac{1}{n^2}s^{\frac{1}{n}-1}u_r(s^\frac{1}{n},t),
  \\
      z_s(s,t)&=\frac{1}{n}v(s^\frac{1}{n},t),
      \quad
      z_{ss}(s,t)=\frac{1}{n^2}s^{\frac{1}{n}-1}v_r(s^\frac{1}{n},t)
  \end{align*}
for all $s\in(0,R^n)$ and $t\in(0,\tmax)$ as well as, with $K$ and $T$ from \eqref{power},
  \begin{align}\label{wspower}
    nw_s(s,t)\le Ks^{-\frac{p}{n}}\quad\mbox{for all}\ s\in(0,R^n]\ \mbox{and}\ t\in(0,T).
  \end{align}
\end{lem}
In order to obtain a key inequality in \eqref{keyineq}, we first prove the following lemma.
%
%
\begin{lem}
Let $\gamma\in(0,1)$ and $s_0\in(0,R^n)$. Then the function $\phi(s_0,\cdot)$
defined as 
\eqref{phi} belongs to $C^0([0,\tmax))\cap C^1((0,\tmax))$ and satisfies
  \begin{align}\label{paphi}
    \frac{\pa\phi}{\pa t}(s_0,t)
      &\ge \chi n\int^{s_0}_{0}s^{-\gamma}(s_0-s)(nw_s(s,t)+1)^{\alpha-1}w(s,t)w_s(s,t)\,ds
    \\ \notag
    &\quad\,
           + n^2m\int^{s_0}_{0}s^{2-\frac{2}{n}-\gamma}(s_0-s)(nw_s(s,t)+1)^{m-1}w_{ss}(s,t)\,ds
    \\ \notag
    &\quad\,
           - \chi n\int^{s_0}_{0}s^{-\gamma}(s_0-s)(nw_s(s,t)+1)^{\alpha-1}z(s,t)w_s(s,t)\,ds
    \\ \notag
    &\quad\,
           - n^{\kappa-1}\mu_1\int^{s_0}_{0}s^{-\gamma}(s_0-s)\left\{\int^{s_0}_{0}
              \sigma^\frac{q}{n}w_s^\kappa(\sigma,t)\,d\sigma\right\}\,ds
    \\ \notag
         &=: I_1(s_0,t) + I_2(s_0,t) + I_3(s_0,t) + I_4(s_0,t)
  \end{align}
for all $t\in(0,\tmax)$.
\end{lem}
%
\begin{proof}
By an argument similar to that in the proof of \cite[Lemma 4.1]{W-2018}
we can show that $\phi(s_0,\cdot)\in C^0([0,\tmax))\cap C^1((0,\tmax))$.
From the second equation in \eqref{PE} we have that 
  \begin{align}\label{second}
    r^{n-1}v_r(r,t)=z(r^n,t)-w(r^n,t)
  \end{align}
for all $r\in(0,R)$ and $t\in(0,\tmax)$.
Noting that $\lambda\ge0$ and \eqref{mu}, we see from \eqref{second} and the first equation in \eqref{PE} that
  \begin{align}\label{wt}
    w_t&\ge n^2ms^{2-\frac{2}{n}}(nw_s+1)^{m-1}w_{ss}
            \\ \notag
            &\quad\,
                +\chi nw_s(nw_s+1)^{\alpha-1}(w-z)
                -n^{\kappa-1}\mu_1\int^{s}_{0}\sigma^\frac{q}{n}w_s^{\kappa}
(\sigma,t)\,d\sigma
  \end{align}
for all $s\in(0,R^n)$ and $t\in(0,\tmax)$.
Thanks to \eqref{phi} and \eqref{wt}, we attain \eqref{paphi}.
\end{proof}
Next we derive an estimate for $I_1$ on the right-hand side of \eqref{paphi}.
%
%
%
\begin{lem}\label{I1}
Let $\gamma\in(0,1)$ and let $\alpha>0$, $\chi>0$ and $p\ge n$ 
and suppose that \eqref{power} holds with some $K>0$ and $T>0$.
Then there exists $C=C(R,\chi,\alpha,p,K)>0$ such that for any $s_0\in(0,R^n)$
  \[
    I_1(s_0,t)\ge C\psi_\alpha(s_0,t)
  \]
for all $t\in(0,{\rm min}\{T,\tmax\})$.
\end{lem}
%
\begin{proof}
In the case $0<\alpha<1$, using \eqref{wspower} and $s<R^n$, we can establish that
  \[
    (nw_s+1)^{\alpha-1}\ge(Ks^{-\frac{p}{n}}+1)^{-(1-\alpha)}
                                \ge(K+R^p)^{-(1-\alpha)}s^{\frac{p}{n}(1-\alpha)}
  \]
for all $s\in(0,s_0)$.
On the other hand, in the case $\alpha\ge1$ it follows that 
  \[
    (nw_s+1)^{\alpha-1}\ge1
  \]
for all $s\in(0,s_0)$. Thus we obtain that
  \begin{align*}
    I_1(s_0,t)&=\chi n\int^{s_0}_{0}s^{\gamma}(s_0-s)(nw_s+1)^{\alpha-1}ww_s\,ds
  \\ \notag
                &\ge \chi n(K+R^p)^{-(1-\alpha)_{+}}\int^{s_0}_{0}s^{-\gamma+\frac{p}{n}(1-\alpha)_{+}}(s_0-s)ww_s\,ds
  \end{align*}
for all $t\in(0,{\rm min}\{T,\tmax\})$, which concludes the proof.
\end{proof}
In order to show estimates for $I_2$, $I_3$ and $I_4$ on the right-hand side of \eqref{paphi}, we introduce two lemmas.
The following lemma has already been proved in \cite[Lemma 3.3]{B-F-L}.
%
%
%
\begin{lem}\label{betaf}
For all $a>-1$ and $b>-1$ and any $s_0\le0$ we have
 \[
    \int^{s_0}_{0}s^a(s_0-s)^b\,ds=B(a+1,b+1)s_0^{a+b+1},
 \]
where $B$ is Euler's beta function.
\end{lem}
%
%
%
%
\begin{lem}\label{lem;wpsi}
Let $\alpha>0$ and $p\ge n$. Assume that $\gamma\in(0,1)$ satisfies that
  \begin{align}\label{wpsicondition}
    \gamma-\frac{p}{n}(1-\alpha)_{+}\in(0,1).
  \end{align}
Then for any $s_0\in(0,R^n)$
  \begin{align}\label{wpsi}
    w(s,t)\le\sqrt{2}s^{\frac{\gamma}{2}-\frac{p}{2n}(1-\alpha)_{+}}(s_0-s)^{-\frac{1}{2}}\sqrt{\psi_{\alpha}(s_0,t)}
  \end{align}
for all $s\in(0,s_0)$ and $t\in(0,\tmax)$.
\end{lem}
%
\begin{proof}
By using the function
  \[
    \psi(s):=\frac{1}{2}s^{-\gamma+\frac{p}{n}(1-\alpha)_{+}}(s_0-s)w^2(s,t)\quad\mbox{for all}\ t\in(0,\tmax)
  \]
instead of $\psi$ in the proof of \cite[Lemma 4.2]{W-2018}, 
we can verify that \eqref{wpsi} holds.
\end{proof}
We establish an estimate for $I_4$.
%
%
%
\begin{lem}\label{I4}
Let $\alpha>0$, $\mu_1>0$, $\kappa\ge1$, $p\ge n$ and $q\ge0$ and suppose that \eqref{power} holds with some $K>0$ and $T>0$. 
Assume that $\gamma\in(0,1)$ satisfies \eqref{wpsicondition} and 
  \begin{align}\label{I4condition}
    \frac{p}{n}[2(\kappa-1)+(1-\alpha)_{+}]-\frac{2q}{n}<\gamma.
  \end{align}
Then there exists $C=C(\gamma,\alpha,\mu_1,\kappa,p,q,K)>0$ such that for any $s_0\in(0,R^n)$
  \begin{align}\label{I4esti}
    I_4(s_0,t)
    \ge-Cs_0^{\frac{3-\gamma}{2}+\frac{q}{n}-\frac{p}{2n}[2(\kappa-1)+(1-\alpha)_{+}]}\sqrt{\psi_\alpha(s_0,t)}
  \end{align} 
for all $t\in(0,{\rm min}\{T,\tmax\})$.
\end{lem}
%
\begin{proof}
Aided by an argument similar to that in the proof of \cite[Lemma 3.5]{B-F-L},
we have from straightforward calculations that
  \begin{align}\label{I4-1}
    I_4(s_0,t)\ge - \frac{n^{\kappa-1}\mu_1}{1-\gamma}s_0^{1-\gamma}
                         \int^{s_0}_{0}s^\frac{q}{n}(s_0-s)w_s^\kappa\,ds
  \end{align}
for all $t\in(0,\tmax)$ and 
  \begin{align}\label{I4-2}
    \int^{s_0}_{0}s^\frac{q}{n}(s_0-s)w_s^\kappa\,ds
    \le c_1s_0\int^{s_0}_{0}s^{\frac{q}{n}-\frac{p}{n}(\kappa-1)-1}w\,ds
  \end{align}
for all $t\in(0,{\rm min}\{T,\tmax\})$, where $c_1:=\frac{K^{\kappa-1}}{n^{\kappa-1}}\left[\left(\frac{p}{n}(\kappa-1)-\frac{q}{n}\right)_{+}+1\right]$.
By virtue of 
Lemmas \ref{betaf} and \ref{lem;wpsi} 
it follows that
  \begin{align}\label{I4-3}
    &c_1s_0\int^{s_0}_{0}s^{\frac{q}{n}-\frac{p}{n}(\kappa-1)-1}w\,ds
  \\ \notag
  &\quad\,
    \le \sqrt{2}c_1s_0\int^{s_0}_{0}
         s^{\frac{q}{n}-\frac{p}{n}(\kappa-1)+\frac{\gamma}{2}
             -\frac{p}{2n}(1-\alpha)_{+}-1}
         (s_0-s)^{-\frac{1}{2}}\,ds\sqrt{\psi_\alpha(s_0,t)}
  \\ \notag
  &\quad\,
      = \sqrt{2}c_1c_2s_0^{\frac{1}{2}+\frac{\gamma}{2}+\frac{q}{n}-\frac{p}{2n}[2(\kappa-1)+(1-\alpha)_{+}]}\sqrt{\psi_\alpha(s_0,t)}
  \end{align}
for all $t\in(0,{\rm min}\{T,\tmax\})$, where 
$c_2:=B\left(\frac{\gamma}{2}+\frac{q}{n}-\frac{p}{2n}[2(\kappa-1)+(1-\alpha)_{+}],\frac{1}{2}\right)$.
Now, noting from \eqref{I4condition} that
  \begin{align*}
    \frac{\gamma}{2}+\frac{q}{n}-\frac{p}{2n}[2(\kappa-1)+(1-\alpha)_{+}]>0,
  \end{align*}
we see that $c_2<\infty$. 
A combination of \eqref{I4-1}, \eqref{I4-2} and \eqref{I4-3} yields \eqref{I4esti}.
\end{proof}
Next we show an estimate for $I_2$.
%
%
%
\begin{lem}\label{I2}
Let $\alpha>0$, $m>0$ and $p\ge n$ and suppose that \eqref{power} holds with some $K>0$ and $T>0$ and  that $\gamma\in(0,1)$ satisfies \eqref{wpsicondition}. 
\begin{itemize}
\item[{\rm (i)}]
  Assume that 
    \[
      0<m<1+\frac{n-2}{p}
    \]
  and
    \begin{align}\label{I2condition1}
      1-\frac{2}{n}-\frac{p}{n}(m-1)_{+}<\gamma
      <2-\frac{4}{n}-\frac{p}{n}[2(m-1)_{+}+(1-\alpha)_{+}].
    \end{align}
  Then there exists $C>0$ such that for any $s_0\in(0,R^n)$
    \[
      I_2(s_0,t)\ge-Cs_0^{\frac{3-\gamma}{2}-\frac{2}{n}-\frac{p}{2n}[2(m-1)_{+}+(1-\alpha)_{+}]}
                         \sqrt{\psi_\alpha(s_0,t)}
                       -Cs_0^{3-\frac{2}{n}-\gamma}
    \]
  for all $t\in(0,{\rm min}\{T,\tmax\})$.
\item[{\rm (ii)}]
  Assume that
    \[
      0<m<{\rm min}\left\{1,\frac{2(n-1)}{p}\right\}
    \]
  and
    \begin{align}\label{I2condition2}
      0<\gamma<2-\frac{2}{n}-\frac{pm}{n}.
    \end{align}
  Then there exists $C>0$ such that for any $s_0\in(0,R^n)$
    \[
      I_2(s_0,t)\ge-Cs_0^{3-\gamma-\frac{2}{n}-\frac{pm}{n}}-Cs_0^{3-\frac{2}{n}-\gamma}
    \]
  for all $t\in(0,{\rm min}\{T,\tmax\})$.
\end{itemize}
\end{lem}
%
\begin{proof}
An argument similar to that in the proof of \cite[Lemma 3.6]{B-F-L} implies the conclusion of this lemma.
\end{proof}
The following lemma has already been proved in the proof of \cite[Lemma 4.6]{W-2018}.
Thus we only recall the statement of the lemma.
%
%
%
\begin{lem}\label{intint}
Let $a\in(1,2)$ and $b\in(0,1)$. Then there exists $C=C(a,b)>0$ such that if $s_0>0$, then
  \[
    \int^{s_0}_{0}\int^{s_0}_{\sigma}\xi^{-a}(s_0-\xi)^{-b}\,d\xi d\sigma
    \le Cs_0^{-b}s^{2-a}
  \]
for all $s\in(0,s_0)$.
\end{lem}
We establish an estimate for $I_3$.
The proof of the following lemma is based on that of \cite[Lemma 3.9]{B-F-L}.
%
%
%
\begin{lem}\label{I3}
Let $\alpha>0$, $\chi>0$, $p\ge n$, $M_0>0$, $K>0$ and $T>0$ and suppose that $\gamma\in(0,1)$ satisfies \eqref{wpsicondition}. 
Assume that 
  \begin{align}\label{I3condition}
    1-\frac{2}{p}<\alpha<1+\frac{2}{p}
    \quad\mbox{and}\quad
    \gamma<2-\frac{2p}{n}(\alpha-1)_{+}.
  \end{align}
Then there exists $C>0$ such that
if $u_0$ fulfills $\int_\Omega u_0=M_0$ 
and 
\eqref{power} holds,
then for any $s_0\in(0,R^n)$
  \begin{align}\label{I3esti}
    I_3(s_0,t)
    &\ge - Cs_0^{\frac{2}{n}+\frac{1-\gamma}{2}-\frac{p}{2n}[(1-\alpha)_{+}+2(\alpha-1)_{+}]}
             \sqrt{\psi_{\alpha}(s_0,t)}
    \\ \notag
    &\quad\,
          - Cs_0^{\frac{2}{n}-\frac{p}{n}[(1-\alpha)_{+}+(\alpha-1)_{+}]}
             \psi_{\alpha}(s_0,t)
  \end{align}
for all $t\in(0,{\rm min}\{T,\tmax\})$.
\end{lem}
%
\begin{proof}
By an argument similar to that in the proof of \cite[Lemma 4.7]{W-2018} we can see that there exists $c_1=c_1(R,\lambda,M_0)>0$ such that
  \begin{align}\label{zesti}
    z\le \frac{c_1}{n}s_0^{\frac{2}{n}-1}s
           + \frac{1}{n^2}\int^{s}_{0}\int^{s_0}_{\sigma}\xi^{\frac{2}{n}-2}w(\xi,t)\,d\xi d\sigma
  \end{align}
for all $s\in(0,s_0)$ and $t\in(0,{\rm min}\{T,\tmax\})$.
First we show the estimate \eqref{I3esti} in the case $1-\frac{2}{n}<\alpha<1$.
Since $1-\frac{2}{n}<\alpha$ and $\gamma<1$, it follows that
  \[
    \gamma-\frac{4}{n}+\frac{2p}{n}(1-\alpha)<\gamma
  \]
and 
  \[
      \left(2-\frac{4}{n}+\frac{p}{n}(1-\alpha)\right)
      - \left(\gamma-\frac{4}{n}+\frac{2p}{n}(1-\alpha)\right)
    =2-\frac{p}{n}(1-\alpha)-\gamma
    >2-\frac{2}{n}-\gamma>0.
  \]
Moreover,
we have from \eqref{wpsicondition} that $\gamma>\frac{p}{n}(1-\alpha)$.
Thus we can take
  \[
    \widetilde{\gamma}
    \in\left({\rm max}\left\{\frac{p}{n}(1-\alpha),\gamma-\frac{4}{n}+\frac{2p}{n}(1-\alpha)
    \right\},{\rm min}\left\{\gamma,2-\frac{4}{n}+\frac{p}{n}(1-\alpha)\right\}\right).
  \]
Noticing that $\widetilde{\gamma}-\frac{p}{n}(1-\alpha)\in(0,1)$, from Lemma \ref{lem;wpsi} we obtain that
  \begin{align*}
    z&\le \frac{c_1}{n}s_0^{\frac{2}{n}-1}s
  \\ \notag
  &\quad\,
           + \frac{\sqrt{2}}{n^2}\int^{s}_{0}
              \int^{s_0}_{\sigma}
               \xi^{\frac{2}{n}-2+\frac{\widetilde{\gamma}}{2}-\frac{p}{n}(1-\alpha)}
               (s_0-\xi)^{-\frac{1}{2}}\,d\xi d\sigma
              \left\{
                \int^{s_0}_{0}s^{-\widetilde{\gamma}+\frac{p}{n}(1-\alpha)}(s_0-s)ww_s\,ds
              \right\}^\frac{1}{2}
  \end{align*}
for all $s\in(0,s_0)$ and $t\in(0,{\rm min}\{T,\tmax\})$.
Since we know from the conditions $1-\frac{2}{p}<\alpha$ and $\widetilde{\gamma}<2-\frac{4}{n}+\frac{p}{n}(1-\alpha)$ that 
  \[
    1<2-\frac{2}{n}-\frac{\widetilde{\gamma}}{2}+\frac{p}{2n}(1-\alpha)<2,
  \]
from Lemma \ref{intint} we can find $c_2=c_2(\gamma,\alpha,p)>0$
such that 
  \begin{align}\label{z-1} 
    z&\le \frac{c_1}{n}s_0^{\frac{2}{n}-1}s
             + c_2s_0^{-\frac{1}{2}}
                s^{\frac{2}{n}+\frac{\widetilde{\gamma}}{2}-\frac{p}{2n}(1-\alpha)}
                \left\{
                  \int^{s_0}_{0}s^{-\widetilde{\gamma}+\frac{p}{n}(1-\alpha)}(s_0-s)ww_s\,ds
                \right\}^\frac{1}{2}
  \\ \notag
      &\le \frac{c_1}{n}s_0^{\frac{2}{n}-1}s
             + c_2s_0^{-\frac{1}{2}+\frac{\gamma-\widetilde{\gamma}}{2}}
                s^{\frac{2}{n}+\frac{\widetilde{\gamma}}{2}-\frac{p}{2n}(1-\alpha)}
                \sqrt{\psi_\alpha(s_0,t)}
  \end{align}
for all $s\in(0,s_0)$ and $t\in(0,{\rm min}\{T,\tmax\})$.
Now we have from the fact $(nw_s+1)^{\alpha-1}\le1$ that
  \begin{align}\label{pro;I3-1}
    I_3(s_0,t)&= - \chi n\int^{s_0}_{0}s^{-\gamma}(s_0-s)(nw_s+1)^{\alpha-1}
zw_s\,ds
  \\ \notag
               &\ge- \chi n\int^{s_0}_{0}s^{-\gamma}(s_0-s)zw_s\,ds
  \end{align}
for all $t\in(0,{\rm min}\{T,\tmax\})$.
Furthermore, by an argument similar to that in the proof of \cite[Lemma 4.1]{W-2018} and \eqref{z-1} we see that
  \begin{align}\label{pro;I3-1-2}
    &- \chi n\int^{s_0}_{0}s^{-\gamma}(s_0-s)zw_s\,ds
  \\ \notag
  &\quad\,
    \ge - \chi n(\gamma+1)s_0\int^{s_0}_{0}s^{-\gamma-1}zw\,ds
  \\ \notag
  &\quad\,
    \ge - c_1\chi(\gamma+1)s_0^\frac{2}{n}\int^{s_0}_{0}s^{-\gamma}w\,ds
  \\ \notag
  &\quad\,\quad\,
          - c_2\chi n(\gamma+1)s_0^{\frac{1}{2}+\frac{\gamma-\widetilde{\gamma}}{2}}
               \int^{s_0}_{0}
                s^{\frac{2}{n}-\gamma+\frac{\widetilde{\gamma}}{2}-\frac{p}{2n}(1-\alpha)-1}
                w\,ds
               \sqrt{\psi_\alpha(s_0,t)}
  \end{align}
for all $t\in(0,{\rm min}\{T,\tmax\})$.
Since the conditions 
$1-\frac{2}{p}<\alpha$ and $\widetilde{\gamma}>\gamma-\frac{4}{n}+\frac{2p}{n}(1-\alpha)$ imply that
  \[
    1-\frac{\gamma}{2}-\frac{p}{2n}(1-\alpha)>1-\frac{\gamma}{2}-\frac{1}{n}>0
  \]
and 
  \[
    \frac{2}{n}-\frac{\gamma-\widetilde{\gamma}}{2}-\frac{p}{n}(1-\alpha)
    >\frac{2}{n}-\frac{2}{n}+\frac{p}{n}(1-\alpha)-\frac{p}{n}(1-\alpha)=0,
  \]
respectively,
we obtain from Lemmas \ref{betaf} and \ref{lem;wpsi} 
that
  \begin{align}
    s_0^\frac{2}{n}\int^{s_0}_{0}s^{-\gamma}w\,ds
    &\le \sqrt{2}s_0^\frac{2}{n}
           \int^{s_0}_{0}s^{-\frac{\gamma}{2}-\frac{p}{2n}(1-\alpha)}
            (s_0-s)^{-\frac{1}{2}}\,ds
           \sqrt{\psi_\alpha(s_0,t)}
  \\ \notag
    &=   \sqrt{2}B\left(1-\frac{\gamma}{2}-\frac{p}{2n}(1-\alpha),\frac{1}{2}\right)
           s_0^{\frac{1-\gamma}{2}+\frac{2}{n}-\frac{p}{2n}(1-\alpha)}\sqrt{\psi_\alpha(s_0,t)}
  \end{align}
and 
  \begin{align}\label{pro;I3-1-last}
    &s_0^{\frac{1}{2}+\frac{\gamma-\widetilde{\gamma}}{2}}
      \int^{s_0}_{0}s^{\frac{2}{n}-\gamma+\frac{\widetilde{\gamma}}{2}-\frac{p}{2n}(1-\alpha)-1}
       w\,ds\sqrt{\psi_\alpha(s_0,t)}
  \\ \notag
  &\quad\,
    \le \sqrt{2}s_0^{\frac{1}{2}+\frac{\gamma-\widetilde{\gamma}}{2}}
         \int^{s_0}_{0}
          s^{\frac{2}{n}-\frac{\gamma-\widetilde{\gamma}}{2}-\frac{p}{n}(1-\alpha)-1}
          (s_0-s)^{-\frac{1}{2}}\,ds
         \cdot\psi_\alpha(s_0,t)
  \\ \notag
  &\quad\,
    =   \sqrt{2}B\left(\frac{2}{n}-\frac{\gamma-\widetilde{\gamma}}{2}-\frac{p}{n}(1-\alpha),
                             \frac{1}{2}\right)
         s_0^{\frac{2}{n}-\frac{p}{n}(1-\alpha)}\psi_\alpha(s_0,t)
  \end{align}
for all $t\in(0,{\rm min}\{T,\tmax\})$.
A combination of \eqref{pro;I3-1} and \eqref{pro;I3-1-2}--\eqref{pro;I3-1-last} yields \eqref{I3esti}.
Similarly, we next establish the estimate \eqref{I3esti} in the case $1\le\alpha<1+\frac{2}{p}$.
Since $\alpha<1+\frac{2}{p}$ and $\gamma<2-\frac{2p}{n}(\alpha-1)$, we see that
  \[
    \gamma-\frac{4}{n}+\frac{2p}{n}(\alpha-1)<\gamma
  \]
and 
  \[
    \left(2-\frac{4}{n}\right)-\left(\gamma-\frac{4}{n}+\frac{2p}{n}(\alpha-1)\right)
    =2-\frac{2p}{n}(\alpha-1)-\gamma>0.
  \]
Hence we can choose 
  \[
    \widetilde{\gamma}
    \in\left({\rm max}\left\{0,\gamma-\frac{4}{n}+\frac{2p}{n}(\alpha-1)
    \right\},{\rm min}\left\{\gamma,2-\frac{4}{n}\right\}\right).
  \]
From \eqref{zesti}, Lemmas \ref{lem;wpsi} and \ref{intint} we observe that there exists $c_3=c_3(\gamma)>0$ such that
  \begin{align}\label{z-2}
    z&\le \frac{c_1}{n}s_0^{\frac{2}{n}-1}s
             + \frac{\sqrt{2}}{n^2}\int^{s}_{0}
                \int^{s_0}_{\sigma}
                 \xi^{\frac{2}{n}-2+\frac{\widetilde{\gamma}}{2}}
                 (s_0-\xi)^{-\frac{1}{2}}\,d\xi d\sigma
                \left\{
                  \int^{s_0}_{0}s^{-\widetilde{\gamma}}(s_0-s)ww_s\,ds
                \right\}^\frac{1}{2}
  \\ \notag
      &\le \frac{c_1}{n}s_0^{\frac{2}{n}-1}s
             + c_3s_0^{-\frac{1}{2}}s^{\frac{2}{n}+\frac{\widetilde{\gamma}}{2}}
                \left\{
                  \int^{s_0}_{0}s^{-\widetilde{\gamma}}(s_0-s)ww_s\,ds
                \right\}^\frac{1}{2}
  \\ \notag
      &\le \frac{c_1}{n}s_0^{\frac{2}{n}-1}s
             + c_3s_0^{-\frac{1}{2}+\frac{\gamma-\widetilde{\gamma}}{2}}
                s^{\frac{2}{n}+\frac{\widetilde{\gamma}}{2}}
                \sqrt{\psi_\alpha(s_0,t)}
  \end{align}
for all $s\in(0,s_0)$ and $t\in(0,{\rm min}\{T,\tmax\})$.
Thanks to \eqref{wspower}, it follows that
  \begin{align}\label{z-ws}
    (nw_s+1)^{\alpha-1}\le(Ks^{-\frac{p}{n}}+1)^{\alpha-1}\le c_4s^{-\frac{p}{n}(\alpha-1)}
  \end{align}
for all $s\in(0,s_0)$ and $t\in(0,{\rm min}\{T,\tmax\})$, where $c_4:=(K+R^p)^{\alpha-1}$.
Applying \eqref{z-ws} to $I_3$, by an argument similar to that in the proof of \cite[Lemma 4.1]{W-2018} and \eqref{z-2} we see that 
  \begin{align}\label{pro;I3-2}
    I_3(s_0,t)&=  - \chi n\int^{s_0}_{0}s^{-\gamma}(s_0-s)(nw_s+1)^{\alpha-1}
zw_s\,ds
  \\ \notag
               &\ge- \chi nc_4\int^{s_0}_{0}s^{-\gamma-\frac{p}{n}(\alpha-1)}(s_0-s)zw_s\,ds
  \\ \notag
               &\ge- \chi nc_4\left(\gamma+\frac{p}{n}(\alpha-1)+1\right)
                         s_0\int^{s_0}_{0}s^{-\gamma-1-\frac{p}{n}(\alpha-1)}zw\,ds
  \\ \notag
               &\ge- \chi c_1c_5
                         s_0^\frac{2}{n}\int^{s_0}_{0}s^{-\gamma-\frac{p}{n}(\alpha-1)}w\,ds
  \\ \notag
  &\quad\,
                      - \chi nc_3c_5
                         s_0^{\frac{1}{2}+\frac{\gamma-\widetilde{\gamma}}{2}}
                        \int^{s_0}_{0}
                        s^{\frac{2}{n}-\gamma+\frac{\widetilde{\gamma}}{2}-\frac{p}{n}(\alpha-1)-1}
                        w\,ds\sqrt{\psi_\alpha(s_0,t)}
  \end{align}
for all $t\in(0,{\rm min}\{T,\tmax\})$, where $c_5:=c_4\left(\gamma+\frac{p}{n}(\alpha-1)+1\right)$.
Here, noticing from $\gamma<2-\frac{2p}{n}(\alpha-1)$ and $\widetilde{\gamma}>\gamma-\frac{4}{n}+\frac{2p}{n}(\alpha-1)$ that 
  \[
    1-\frac{\gamma}{2}-\frac{p}{n}(\alpha-1)>0
  \] 
and 
  \[
    \frac{2}{n}-\frac{\gamma-\widetilde{\gamma}}{2}-\frac{p}{n}(\alpha-1)
    >\frac{2}{n}-\frac{2}{n}+\frac{p}{n}(\alpha-1)-\frac{p}{n}(\alpha-1)=0,
  \]
we have from Lemmas \ref{betaf} and \ref{lem;wpsi} 
that
  \begin{align}\label{pro;I3-2-2}
    s_0^\frac{2}{n}\int^{s_0}_{0}s^{-\gamma-\frac{p}{n}(\alpha-1)}w\,ds
    &\le \sqrt{2}s_0^\frac{2}{n}
           \int^{s_0}_{0}
            s^{-\frac{\gamma}{2}-\frac{p}{n}(\alpha-1)}(s_0-s)^{-\frac{1}{2}}
            \,ds\sqrt{\psi_\alpha(s_0,t)}
  \\ \notag
    &=    \sqrt{2}B\left(1-\frac{\gamma}{2}-\frac{p}{n}(\alpha-1),\frac{1}{2}\right)
            s_0^{\frac{1-\gamma}{2}+\frac{2}{n}-\frac{p}{n}(\alpha-1)}\sqrt{\psi_\alpha(s_0,t)}
  \end{align}
and 
  \begin{align}\label{pro;I3-2-last}
    &s_0^{\frac{1}{2}+\frac{\gamma-\widetilde{\gamma}}{2}}
      \int^{s_0}_{0}
       s^{\frac{2}{n}-\gamma+\frac{\widetilde{\gamma}}{2}-\frac{p}{n}(\alpha-1)-1}
       w\,ds\sqrt{\psi_\alpha(s_0,t)}
  \\ \notag
  &\quad\,
    \le \sqrt{2}s_0^{\frac{1}{2}+\frac{\gamma-\widetilde{\gamma}}{2}}
         \int^{s_0}_{0}
          s^{\frac{2}{n}-\frac{\gamma-\widetilde{\gamma}}{2}-\frac{p}{n}(\alpha-1)-1}
          (s_0-s)^{-\frac{1}{2}}\,ds\cdot\psi_\alpha(s_0,t)
  \\ \notag
  &\quad\,
    =    \sqrt{2}B\left(\frac{2}{n}-\frac{\gamma-\widetilde{\gamma}}{2}-\frac{p}{n}(\alpha-1),
                              \frac{1}{2}\right)
          s_0^{\frac{2}{n}-\frac{p}{n}(\alpha-1)}\psi_\alpha(s_0,t)
  \end{align}
for all $t\in(0,{\rm min}\{T,\tmax\})$.
A combination of \eqref{pro;I3-2}, \eqref{pro;I3-2-2} and \eqref{pro;I3-2-last} leads to that we can attain \eqref{I3esti}.
\end{proof}
We finally derive an estimate for $\psi_\alpha$.
%
%
%
\begin{lem}\label{phipsi}
Let $\alpha>0$ and $p\ge n$. Suppose that $\gamma\in(0,1)$ satisfies that
  \begin{align}\label{psicondition}
    \gamma<2-\frac{p}{n}(1-\alpha)_{+}.
  \end{align}
Then there exists $C=C(\gamma,\alpha,p)>0$ such that for any $s_0\in(0,R^n)$
  \[
    \phi(s_0,t)\le Cs_0^{\frac{3-\gamma}{2}-\frac{p}{2n}(1-\alpha)_{+}}\sqrt{\psi_\alpha(s_0,t)}
  \]
for all $t\in(0,\tmax)$.
\end{lem}
%
\begin{proof}
From \eqref{psicondition} it follows that 
$1-\frac{\gamma}{2}-\frac{p}{2n}(1-\alpha)_{+}>0$.
Therefore we obtain from Lemmas \ref{betaf} and \ref{lem;wpsi} 
that
  \begin{align*}
    \phi(s_0,t)&=   \int^{s_0}_{0}s^{-\gamma}(s_0-s)w\,ds
  \\ \notag
                  &\le s_0\int^{s_0}_{0}s^{-\gamma}w\,ds
  \\ \notag
                  &\le \sqrt{2}s_0
                         \int^{s_0}_{0}
                          s^{-\frac{\gamma}{2}-\frac{p}{2n}(1-\alpha)_{+}}(s_0-s)^{-\frac{1}{2}}\,ds
                         \sqrt{\psi_\alpha(s_0,t)}
  \\ \notag
                  &=  \sqrt{2}B\left(1-\frac{\gamma}{2}-\frac{p}{2n}(1-\alpha)_{+},\frac{1}{2}\right)
                        s_0^{\frac{3-\gamma}{2}-\frac{p}{2n}(1-\alpha)_{+}}\sqrt{\psi_\alpha(s_0,t)}
  \end{align*}
for all $t\in(0,\tmax)$, which concludes the proof.
\end{proof}
\section{Super-quadratic nonlinear differential inequality for $\bm{\phi}$}
In this section we will derive a differential inequality for the moment type functional $\phi$
by using the pointwise lower estimates for $I_1$, $I_2$, $I_3$ and $I_4$.
Therefore we prove that there exists $\gamma\in(0,1)$ to apply Lemmas \ref{I1}, \ref{I4}, \ref{I2}, \ref{I3} and \ref{phipsi} to \eqref{paphi}.
We define the conditions (A3-1), (A3-2), (B1-1), (B1-2), (C1-1), (C1-2), (C3-1), (C3-2), (C3-3), (D2-1) and (D2-2) as follows:
\begin{itemize}
  \item In the case $n=3$,
    \begin{align*}
      \tag{A3-1} &1-\frac{1}{p}<\alpha<1+\frac{2}{p},\quad
                      \frac{1}{p}\le m<\frac{2}{p},\quad
                      2\alpha-m\le2+\frac{2}{p},
      \\
      \tag{A3-2} &1-\frac{1}{p}<\alpha<1,\quad
                        \frac{2}{p}\le m<\frac{3}{p},\quad
                        m+\alpha<1+\frac{2}{p}.
    \end{align*}
  \item In the case $n=4$,
    \begin{align*}
      \tag{B1-1} &1-\frac{2}{p}<\alpha<1,\quad 
                        0<m<\frac{2}{p},
      \\
      \tag{B1-2} &1\le\alpha<1+\frac{2}{p},\quad 
                        0<m<\frac{2}{p}.
    \end{align*}
  \item In the case $n=5$,
    \begin{align*}
      \tag{C1-1} &1-\frac{2}{p}<\alpha\le1-\frac{1}{p},\quad
                        0<m<\frac{3}{p},
      \\
      \tag{C1-2} &1-\frac{1}{p}<\alpha<1+\frac{2}{p},\quad
                        0<m<1+\frac{1}{2p},\quad
                        2m-\alpha<1+\frac{1}{p},
      \\
      \tag{C3-1}  &1-\frac{2}{p}<\alpha\le1-\frac{1}{p},\quad
                         \frac{3}{p}\le m<1,\quad
                         m+\alpha\ge1+\frac{2}{p},
      \\
      \tag{C3-2} &1-\frac{2}{p}<\alpha<1,\quad
                        1\le m<1+\frac{1}{2p},\quad
                        2m-\alpha\ge1+\frac{1}{p},
      \\
      \tag{C3-3} &1-\frac{2}{p}<\alpha<1+\frac{2}{p},\quad
                        1+\frac{1}{2p}\le m<1+\frac{3}{p},\quad
                        m-\alpha<\frac{3}{p}.
    \end{align*}
  \item In the case $n\ge6$,
    \begin{align*}
      \tag{D2-1}  &1-\frac{2}{p}<\alpha<1+\frac{2}{p},\quad
                         1+\frac{n-6}{2p}\le m<1+\frac{n-4}{2p},\quad
                         2m-\alpha\ge1+\frac{n-4}{p},
      \\
      \tag{D2-2} &1-\frac{2}{p}<\alpha<1+\frac{2}{p},\quad
                        1+\frac{n-4}{2p}\le m<1+\frac{n-2}{p},\quad
                        m-\alpha<\frac{n-2}{p}.
    \end{align*}
\end{itemize}
We first show that there exists $\gamma\in(0,1)$ satisfying 
\eqref{wpsicondition}, \eqref{I4condition}, \eqref{I2condition1}, 
the second condition of \eqref{I3condition} and \eqref{psicondition}.
%
%
%
\begin{lem}\label{gamma1}
Let $m>0$, $\alpha>0$, $\kappa\ge1$, $p\ge n$ and $q\ge0$. 
Assume that $m$ and $\alpha$ satisfy
{\rm (A4)}, {\rm (B1-2)}, {\rm (B3)}, {\rm (C1-2)}, {\rm (C3)}, {\rm (D1)} or {\rm (D2)}.
Suppose that $\kappa$ fulfills 
{\rm ({\bf I\hspace{-.1em}I})} and {\rm ({\bf I\hspace{-.1em}V})}.
Then there exists $\gamma\in(0,1)$ such that
  \begin{align}\label{lem4.1}
    &{\rm max}\left\{
                    \frac{p}{n}(1-\alpha)_{+},
                    \frac{p}{n}[2(\kappa-1)+(1-\alpha)_{+}]-\frac{2q}{n},
                    1-\frac{2}{n}-\frac{p}{n}(m-1)_{+}
                    \right\}
  \\ \notag
    &<\gamma<
      {\rm min}\left\{
                   1,2-\frac{4}{n}-\frac{p}{n}[2(m-1)_{+}+(1-\alpha)_{+}]
                   \right\}.
  \end{align}
\end{lem}
%
\begin{proof}
%
%
We first consider the case that $m$ and $\alpha$ satisfy (A4), (B3), (C3) or (D2).
In the cases that $n=3$ and the condition (A4) holds 
and that $n=4$ and the condition (B3) holds
we see that
  \begin{align*}
    1-\left(2-\frac{4}{n}-\frac{p}{n}[2(m-1)_{+}+(1-\alpha)_{+}]\right)
    &=-1+\frac{4}{n}+\frac{p}{n}[2(m-1)_{+}+(1-\alpha)_{+}]
    \\
    &\ge-1+\frac{4}{n}\ge0.
  \end{align*}
Moreover, thanks to 
the conditions 
$\alpha\le1-\frac{1}{p}$ in the case that $n=5$ and (C3-1) holds, 
$2m-\alpha\ge1+\frac{n-4}{p}$ in the cases that $n=5$ and (C3-2) holds 
and that $n\ge6$ and (D2-1) holds
and $m\ge1+\frac{n-4}{2p}$ in the cases that $n=5$ and (C3-3) holds 
and that $n\ge6$ and (D2-2) holds, 
we obtain that 
  \begin{align*}
    1-\left(2-\frac{4}{n}-\frac{p}{n}[2(m-1)_{+}+(1-\alpha)_{+}]\right)
    &=-1+\frac{4}{n}+\frac{p}{n}[2(m-1)_{+}+(1-\alpha)_{+}]
    \\
    &\ge-1+\frac{4}{n}+\frac{n-4}{n}=0.
  \end{align*}
Thus it suffices to show that 
the following conditions hold:
  \begin{align}
  \label{lem4.1-c1}
    \frac{p}{n}(1-\alpha)_{+}<2-\frac{4}{n}-\frac{p}{n}[2(m-1)_{+}+(1-\alpha)_{+}],
  \\ 
  \label{lem4.1-c2}
    \frac{p}{n}[2(\kappa-1)+(1-\alpha)_{+}]-\frac{2q}{n}
    <2-\frac{4}{n}-\frac{p}{n}[2(m-1)_{+}+(1-\alpha)_{+}],
  \\
  \label{lem4.1-c3}
    1-\frac{2}{n}-\frac{p}{n}(m-1)_{+}<2-\frac{4}{n}-\frac{p}{n}[2(m-1)_{+}+(1-\alpha)_{+}].
  \end{align}
Here we have that
  \begin{align*}
    &\left(2-\frac{4}{n}-\frac{p}{n}[2(m-1)_{+}+(1-\alpha)_{+}]\right)-\frac{p}{n}(1-\alpha)_{+}
    \\ \notag
    &\quad\,
    =2-\frac{4}{n}-\frac{2p}{n}[(m-1)_{+}+(1-\alpha)_{+}].
  \end{align*}
In the cases that $n=3$ and (A4) holds and that $n=4$ and (B3) holds, 
if $m<1$ and $\alpha\ge1$, then it follows that
  \[
    2-\frac{4}{n}-\frac{2p}{n}[(m-1)_{+}+(1-\alpha)_{+}]=2-\frac{4}{n}>0.
  \]
Furthermore, invoking from (A4) and (B3) 
that
  \begin{align*}
      \alpha>1-\frac{n-2}{p} &\qquad\mbox{if}\quad m<1 \quad\mbox{and}\quad \alpha<1,
    \\
      m-\alpha<\frac{n-2}{p} &\qquad\mbox{if}\quad m\ge1 \quad\mbox{and}\quad \alpha<1,
    \\
      m<1+\frac{n-2}{p} &\qquad\mbox{if}\quad m\ge1 \quad\mbox{and}\quad \alpha\ge1,
  \end{align*}
we know that 
  \[
    2-\frac{4}{n}-\frac{2p}{n}[(m-1)_{+}+(1-\alpha)_{+}]>2-\frac{4}{n}-\frac{2(n-2)}{n}=0.
  \]
On the other hand, in the case that $n=5$ and (C3-1) holds we see from the conditions 
$m<1$ and $\alpha>1-\frac{2}{p}$ that
  \[
    2-\frac{4}{n}-\frac{2p}{n}[(m-1)_{+}+(1-\alpha)_{+}]
    >2-\frac{4}{n}-\frac{4}{n}=\frac{2}{5}>0.
  \]
In the cases that $n=5$ and (C3-2) holds and that $n\ge6$ and (D2-1) holds, 
by virtue of the conditions 
$n\ge5$, $m<1+\frac{n-4}{2p}$ and $\alpha>1-\frac{2}{p}$
we obtain that
  \[
    2-\frac{4}{n}-\frac{2p}{n}[(m-1)_{+}+(1-\alpha)_{+}]
    >2-\frac{4}{n}-\frac{2p}{n}\left[\frac{n-4}{2p}+\frac{2}{p}\right]=1-\frac{4}{n}>0.
  \]
In the cases that $n=5$ and (C3-3) holds and that $n\ge6$ and (D2-2) holds, 
recalling $m-\alpha<\frac{n-2}{p}$, 
we can establish that
  \[
    2-\frac{4}{n}-\frac{2p}{n}[(m-1)_{+}+(1-\alpha)_{+}]
    >2-\frac{4}{n}-\frac{2(n-2)}{n}=0.
  \]
Therefore we attain \eqref{lem4.1-c1}.
Moreover, from the fact $2-\frac{4}{n}-\frac{2p}{n}[(m-1)_{+}+(1-\alpha)_{+}]>0$
we deduce that
  \begin{align*}
    &\left(2-\frac{4}{n}-\frac{p}{n}[2(m-1)_{+}+(1-\alpha)_{+}]\right)
    -\left(1-\frac{2}{n}-\frac{p}{n}(m-1)_{+}\right)
  \\ &\quad\,
    =1-\frac{2}{n}-\frac{p}{n}[(m-1)_{+}+(1-\alpha)_{+}]
  \\ &\quad\,
    =\frac{1}{2}\left(2-\frac{4}{n}-\frac{2p}{n}[(m-1)_{+}+(1-\alpha)_{+}]\right)
  \\ &\quad\,
    >0,
  \end{align*}
which implies that \eqref{lem4.1-c3} holds.
Noticing from {\rm ({\bf I\hspace{-.1em}V})} that 
  \[
    \kappa<1+\frac{n-2}{p}+\frac{q}{p}-(m-1)_{+}-(1-\alpha)_{+},
  \]
we have that
  \begin{align*}
    &\left(2-\frac{4}{n}-\frac{p}{n}[2(m-1)_{+}+(1-\alpha)_{+}]\right)
    -\left(\frac{p}{n}[2(\kappa-1)+(1-\alpha)_{+}]-\frac{2q}{n}\right)
  \\ &\quad\,
    =2-\frac{4}{n}+\frac{2q}{n}-\frac{2p}{n}[(\kappa-1)+(m-1)_{+}+(1-\alpha)_{+}]
  \\ &\quad\,
    >2-\frac{4}{n}+\frac{2q}{n}-\frac{2p}{n}\cdot\left[\frac{n-2}{p}+\frac{q}{p}\right]
  \\ &\quad\,
    =0,
  \end{align*}
which attains \eqref{lem4.1-c2}.
Thus, in the cases 
(A4), (B3), (C3) and (D2) 
we can take $\gamma\in(0,1)$ with \eqref{lem4.1}.
%
%
Next we consider the case that $m$ and $\alpha$ fulfill (B1-2).
Since it follows from the conditions $n=4$, $m<1$ and $\alpha\ge1$ that
  \[
    1-\left(2-\frac{4}{n}-\frac{p}{n}[2(m-1)_{+}+(1-\alpha)_{+}]\right)
    =-1+\frac{4}{n}=0,
  \]
we confirm that
  \begin{align} 
  \label{lem4.1-c4}
    \frac{p}{n}(1-\alpha)_{+}<1,
  \\
  \label{lem4.1-c5}
    \frac{p}{n}[2(\kappa-1)+(1-\alpha)_{+}]-\frac{2q}{n}<1,
  \\ 
  \label{lem4.1-c6}
    1-\frac{2}{n}-\frac{p}{n}(m-1)_{+}<1.
  \end{align}
Noting that $(m-1)_{+}=0$ and $(1-\alpha)_{+}=0$,
we see that \eqref{lem4.1-c4} and \eqref{lem4.1-c6} hold.
Furthermore, recalling from {\rm ({\bf I\hspace{-.1em}I})} that
  \[
    \kappa<1+\frac{n}{2p}+\frac{q}{p},
  \]
we obtain that
  \[
    1-\left(\frac{2p}{n}(\kappa-1)-\frac{2q}{n}\right)
    >1-\frac{2p}{n}\left(\frac{n}{2p}+\frac{q}{p}\right)+\frac{2q}{n}=0,
  \]
which infers that we can choose $\gamma\in(0,1)$ with \eqref{lem4.1}.
%
%
Finally we verify that there exists $\gamma\in(0,1)$ with \eqref{lem4.1} in the case that $m$ and $\alpha$ satisfy the condition (C1-2) or (D1).
The conditions $n\ge5$ and $2m-\alpha<1+\frac{n-4}{p}$ yield that
  \begin{align*}
    1-\left(2-\frac{4}{n}-\frac{p}{n}[2(m-1)_{+}+(1-\alpha)_{+}]\right)
    &=-1+\frac{4}{n}+\frac{p}{n}[2(m-1)_{+}+(1-\alpha)_{+}]
  \\
    &<-1+\frac{4}{n}+\frac{n-4}{n}
  \\
    &=0.
  \end{align*}
Thus we show \eqref{lem4.1-c4}--\eqref{lem4.1-c6}.
Since $n\ge5$ and $\alpha>1-\frac{2}{p}$, 
we have that
  \[
    1-\frac{p}{n}(1-\alpha)_{+}>1-\frac{2}{n}>0.
  \]
Moreover, it follows that
  \[
    1-\left(1-\frac{2}{n}-\frac{p}{n}(m-1)_{+}\right)=\frac{2}{n}+\frac{p}{n}(m-1)_{+}>0.
  \]
Invoking from {\rm ({\bf I\hspace{-.1em}I})} that
  \[
    \kappa<1+\frac{n}{2p}+\frac{q}{p}-\frac{(1-\alpha)_{+}}{2},
  \]
we see that
  \[
    1-\left(\frac{p}{n}[2(\kappa-1)+(1-\alpha)_{+}]-\frac{2q}{n}\right)
    >1-\frac{p}{n}\cdot\left[\frac{n}{p}+\frac{2q}{p}\right]+\frac{2q}{n}=0.
  \]
Therefore, since \eqref{lem4.1-c4}--\eqref{lem4.1-c6} hold, 
we can find $\gamma\in(0,1)$ with \eqref{lem4.1}.
\end{proof}
Next we prove that there exists $\gamma\in(0,1)$ such that 
\eqref{wpsicondition}, \eqref{I4condition}, \eqref{I2condition2},
the second condition of \eqref{I3condition} and \eqref{psicondition} hold.
%
%
%
\begin{lem}\label{gamma2}
Let $m>0$, $\alpha>0$, $\kappa\ge1$, $p\ge n$ and $q\ge0$. 
Assume that $m$ and $\alpha$ satisfy
{\rm (A1)}, {\rm (A2)}, {\rm (A3)}, {\rm (B1-1)}, {\rm (B2)}, {\rm (C1-1)} or {\rm (C2)}.
Suppose that $\kappa$ fulfills 
{\rm ({\bf I})},{\rm ({\bf I\hspace{-.1em}I})} and {\rm ({\bf I\hspace{-.1em}I\hspace{-.1em}I})}.
Then there exists $\gamma\in(0,1)$ such that
  \begin{align}\label{lem4.2}
    &{\rm max}\left\{
                    \frac{p}{n}(1-\alpha)_{+},
                    \frac{p}{n}[2(\kappa-1)+(1-\alpha)_{+}]-\frac{2q}{n}
                    \right\}
  \\ \notag
    &<\gamma<
      {\rm min}\left\{
                   1,2-\frac{2}{n}-\frac{pm}{n},2-\frac{2p}{n}(\alpha-1)_{+}
                   \right\}.
  \end{align}
\end{lem}
%
\begin{proof}
%
%
First we consider the case that 
$m$ and $\alpha$ fulfill (A1), (B1-1) or (C1-1).
By virtue of the condition $m<\frac{n-2}{p}\ (n\in\{3,4,5\})$ we obtain that
  \[
    \left(2-\frac{2}{n}-\frac{pm}{n}\right)-1
    =1-\frac{2}{n}-\frac{pm}{n}
    >1-\frac{2}{n}-\frac{n-2}{n}=0.
  \]
Furthermore, in the case that $n=3$ and (A1) holds 
we can estimate from the condition $\alpha<1+\frac{3}{2p}$ that
  \[
    \left(2-\frac{2p}{n}(\alpha-1)_{+}\right)-1
    =1-\frac{2p}{n}(\alpha-1)_{+}
    >1-\frac{3}{n}=0.
  \]
In the cases that $n=4$ and (B1-1) holds and that $n=5$ and (C1-1) holds,
noticing that $(\alpha-1)_{+}=0$, we know that
  \[
    \left(2-\frac{2p}{n}(\alpha-1)_{+}\right)-1=1>0.  
  \]
Hence it suffices to show that 
  \begin{align}
    \label{lem4.2-c1}
    \frac{p}{n}(1-\alpha)_{+}<1,
  \\
    \label{lem4.2-c2}
    \frac{p}{n}[2(\kappa-1)+(1-\alpha)_{+}]-\frac{2q}{n}<1.
  \end{align}
In the case that $n=3$ and (A1) holds we have from the condition $\alpha>1-\frac{1}{p}$ that
  \[
    1-\frac{p}{n}(1-\alpha)_{+}
    >1-\frac{1}{n}>0.
  \]
On the other hand, 
in the cases that $n=4$ and (B1-1) holds and that $n=5$ and (C1-1) holds
we see from the condition $\alpha>1-\frac{2}{p}$ that
  \[
    1-\frac{p}{n}(1-\alpha)_{+}
    >1-\frac{2}{n}>0.
  \]
Thus the condition \eqref{lem4.2-c1} holds.
Recalling from {\rm ({\bf I\hspace{-.1em}I})} that
  \[
    \kappa<1+\frac{n}{2p}+\frac{q}{p}-\frac{(1-\alpha)_{+}}{2},
  \]
we can show that
  \[
    1-\left(\frac{p}{n}[2(\kappa-1)+(1-\alpha)_{+}]-\frac{2q}{n}\right)
    >1-\frac{p}{n}\cdot\left[\frac{n}{p}+\frac{2q}{p}\right]+\frac{2q}{n}=0,
  \]
which implies that we attain \eqref{lem4.2-c2}.
Since \eqref{lem4.2-c1} and \eqref{lem4.2-c2} hold, we can take $\gamma\in(0,1)$ with \eqref{lem4.2}.
%
%
Next we consider the case that $m$ and $\alpha$ satisfy (A2).
From the conditions $n=3$ and $2\alpha-m>2+\frac{2}{p}$ we obtain that
  \[
    \left(2-\frac{2}{n}-\frac{pm}{n}\right)-\left(2-\frac{2p}{n}(\alpha-1)_{+}\right)
    =-\frac{2}{n}+\frac{p}{n}(2\alpha-m-2)
    >-\frac{2}{n}+\frac{2}{n}
    =0.
  \]
Moreover, the condition $\alpha\ge1+\frac{3}{2p}$ yields that
  \[
    1-\left(2-\frac{2p}{n}(\alpha-1)_{+}\right)
    =-1+\frac{2p}{n}(\alpha-1)
    \ge-1+\frac{3}{n}=0.
  \]
Accordingly, we confirm that 
  \begin{align}
    \label{lem4.2-c3}
    \frac{p}{n}(1-\alpha)_{+}<2-\frac{2p}{n}(\alpha-1)_{+},
  \\
    \label{lem4.2-c4}
    \frac{p}{n}[2(\kappa-1)+(1-\alpha)_{+}]-\frac{2q}{n}<2-\frac{2p}{n}(\alpha-1)_{+}.
  \end{align}
Since $\frac{p}{n}(1-\alpha)_{+}=0$ 
and it follows 
from the condition $\alpha<1+\frac{2}{p}$
that 
  \[
    2-\frac{2p}{n}(\alpha-1)_{+}>2-\frac{4}{n}>0,
  \]
we can verify that \eqref{lem4.2-c3} holds. 
Noticing from {\rm ({\bf I})} that
  \[
    \kappa<1+\frac{3}{p}+\frac{q}{p}-(\alpha-1),
  \]
we see that
  \begin{align*}
    \left(2-\frac{2p}{n}(\alpha-1)_{+}\right)
    -\left(\frac{p}{n}[2(\kappa-1)+(1-\alpha)_{+}]-\frac{2q}{n}\right)
    &=2-\frac{2p}{n}[(\kappa-1)+(\alpha-1)]+\frac{2q}{n}
    \\
    &>2-\frac{2p}{n}\cdot\left[\frac{3}{p}+\frac{q}{p}\right]+\frac{2q}{n}
    \\
    &=0,
  \end{align*}
which infers \eqref{lem4.2-c4}. 
Consequently, we can find $\gamma\in(0,1)$ with \eqref{lem4.2}.
%
%
Finally we consider the case that $m$ and $\alpha$ fulfill (A3), (B2) or (C2).
In light of the condition $2\alpha-m\le2+\frac{2}{p}$ in the case that $n=3$ and (A3-1) holds 
and the condition $(\alpha-1)_{+}=0$ in the cases that $n=3$ and (A3-2) holds,
that $n=4$ and (B2) holds and that $n=5$ and (C2) holds
we have that 
  \[
    \left(2-\frac{2p}{n}(\alpha-1)_{+}\right)-\left(2-\frac{2}{n}-\frac{pm}{n}\right)
    =\frac{2}{n}-\frac{p}{n}[2(\alpha-1)_{+}-m]
    \ge0.
  \]
Moreover, since $m\ge\frac{n-2}{p}$ in the cases that $n=3$ and (A3) holds, 
that $n=4$ and (B2) holds and that $n=5$ and (C2) holds, it follows that
  \[
    1-\left(2-\frac{2}{n}-\frac{pm}{n}\right)
    =-1+\frac{2}{n}+\frac{pm}{n}
    \ge-1+\frac{2}{n}
+\frac{n-2}{n}=0.
  \]
Therefore it suffices to show that
  \begin{align}
    \label{lem4.2-c5}
    \frac{p}{n}(1-\alpha)_{+}<2-\frac{2}{n}-\frac{pm}{n},
  \\
    \label{lem4.2-c6}
    \frac{p}{n}[2(\kappa-1)+(1-\alpha)_{+}]-\frac{2q}{n}<2-\frac{2}{n}-\frac{pm}{n}.
  \end{align}
Now we know that
  \[
    \left(2-\frac{2}{n}-\frac{pm}{n}\right)-\frac{p}{n}(1-\alpha)_{+}
    =2-\frac{2}{n}-\frac{p}{n}[m+(1-\alpha)_{+}].
  \]
In the case that $n=3$ and (A3) holds, 
thanks to the conditions $m<\frac{3}{p}$ and $\alpha>1-\frac{1}{p}$, we obtain that
  \[
    2-\frac{2}{n}-\frac{p}{n}[m+(1-\alpha)_{+}]
    >2-\frac{2}{n}-\frac{p}{n}\cdot\left[\frac{3}{p}+\frac{1}{p}\right]=0.
  \]
On the other hand, in the cases that $n=4$ and (B2) holds and that $n=5$ and (C2) holds
we deduce from the conditions $m<\frac{4}{p}$ and $\alpha>1-\frac{2}{p}$ that
  \[
    2-\frac{2}{n}-\frac{p}{n}[m+(1-\alpha)_{+}]
    >2-\frac{2}{n}-\frac{p}{n}\cdot\left[\frac{4}{p}+\frac{2}{p}\right]
    \ge0.
  \]
Thus the condition 
\eqref{lem4.2-c5} holds.
Invoking from {\rm ({\bf I\hspace{-.1em}I\hspace{-.1em}I})} that
  \[
    \kappa<1+\frac{n-1}{p}+\frac{q}{p}-\frac{m}{2}-\frac{(1-\alpha)_{+}}{2},
  \]
we can show that
  \begin{align*}
    &\left(2-\frac{2}{n}-\frac{pm}{n}\right)
    -\left(\frac{p}{n}[2(\kappa-1)+(1-\alpha)_{+}]-\frac{2q}{n}\right)
  \\ &\quad\,
    =2-\frac{2}{n}-\frac{p}{n}[2(\kappa-1)+m+(1-\alpha)_{+}]+\frac{2q}{n}
  \\ &\quad\,
    >2-\frac{2}{n}-\frac{p}{n}\cdot\left[\frac{n-1}{p}+\frac{q}{p}\right]+\frac{2q}{n}
  \\ &\quad\,
    =0,
  \end{align*}
which infers that \eqref{lem4.2-c6} holds.
Accordingly, we can choose $\gamma\in(0,1)$ with \eqref{lem4.2}.
\end{proof}
Thanks to Lemmas \ref{gamma1} and \ref{gamma2}, we can apply Lemmas \ref{I1}, \ref{I4},  \ref{I2}, \ref{I3} and \ref{phipsi} to  \eqref{paphi}.
Therefore, we finally establish a super-quadratic nonlinear differential inequality for the moment-type functional $\phi$ defined as \eqref{phi}.
%
%
%
\begin{lem}\label{keylem}
Let $m>0$, $\alpha>0$, $\chi>0$, $\mu_1>0$, $\kappa\ge1$, $p\ge n$, $q\ge0$,
$M_0>0$, $K>0$ and $T>0$.
  \begin{itemize}
  \item[{\rm (i)}] 
  Assume that $m$ and $\alpha$ satisfy 
  {\rm (A4)}, {\rm (B1-2)}, {\rm (B3)}, {\rm (C1-2)}, {\rm (C3)}, {\rm (D1)} or {\rm (D2)}
  and $\kappa$ fulfills 
  {\rm ({\bf I\hspace{-.1em}I})} and {\rm ({\bf I\hspace{-.1em}V})}.
  Then there exist $C>0$, $\gamma\in(0,1)$, 
  $\theta\in\left(0,2-\frac{p}{n}(1-\alpha)_{+}\right)$ and $s_1\in(0,R^n)$
  such that if $u_0$ satisfies $\int_\Omega u_0=M_0$ and \eqref{power} holds, then
    \begin{align}\label{keyineq}
      \frac{\pa \phi}{\pa t}(s_0,t)
      \ge \frac{1}{C}s_0^{\gamma-3+\frac{p}{n}(1-\alpha)_{+}}\phi^2(s_0,t)
            -Cs_0^{3-\gamma-\theta}
    \end{align}
  for all $s_0\in(0,s_1)$ and $t\in(0,{\rm min}\{T,\tmax\})$.
  \item[{\rm (ii)}]
  Assume that $m$ and $\alpha$ satisfy 
  {\rm (A1)}, {\rm (A2)}, {\rm (A3)}, {\rm (B1-1)}, {\rm (B2)}, {\rm (C1-1)} or {\rm (C2)}
  and $\kappa$ fulfills 
  {\rm ({\bf I})}, {\rm ({\bf I\hspace{-.1em}I})} 
  and {\rm ({\bf I\hspace{-.1em}I\hspace{-.1em}I})}.
  Then there exist $C>0$, $\gamma\in(0,1)$, 
  $\theta\in\left(0,2-\frac{p}{n}(1-\alpha)_{+}\right)$ and $s_1\in(0,R^n)$
  such that if $u_0$ satisfies $\int_\Omega u_0=M_0$ and \eqref{power} holds, then 
  \eqref{keyineq} holds for all $s_0\in(0,s_1)$ and $t\in(0,{\rm min}\{T,\tmax\})$.
  \end{itemize}
\end{lem}
%
\begin{proof}
%
%
We first prove \eqref{keyineq} in the case that $m$ and $\alpha$ satisfy 
{\rm (A4)}, {\rm (B1-2)}, {\rm (B3)}, {\rm (C1-2)}, {\rm (C3)}, {\rm (D1)} or {\rm (D2)}.
By virtue of Lemma \ref{gamma1} we can take $\gamma\in(0,1)$ fulfilling \eqref{lem4.1}. 
Thus, applying
Lemma \ref{I1}, Lemma \ref{I4}, part (i) of Lemma \ref{I2} and  Lemma \ref{I3} 
to \eqref{paphi}, we obtain that there exist $c_1>0$ and $c_2>0$ such that
  \begin{align*}
    \frac{\pa \phi}{\pa t}(s_0,t)
    &\ge c_1\psi_\alpha(s_0,t)
    \\ \notag
    &\quad\,
           -c_2s_0^{\frac{3-\gamma}{2}-\frac{2}{n}-\frac{p}{2n}[2(m-1)_{+}+(1-\alpha)_{+}]}
                        \sqrt{\psi_\alpha(s_0,t)}
    \\ \notag
    &\quad\,
           -c_2s_0^{3-\frac{2}{n}-\gamma}
    \\ \notag
    &\quad\,
           -c_2s_0^{\frac{2}{n}+\frac{1-\gamma}{2}-\frac{p}{2n}[(1-\alpha)_{+}+2(\alpha-1)_{+}]}
                        \sqrt{\psi_{\alpha}(s_0,t)}
    \\ \notag
    &\quad\,
            -c_2s_0^{\frac{2}{n}-\frac{p}{n}[(1-\alpha)_{+}+(\alpha-1)_{+}]}\psi_{\alpha}(s_0,t)
    \\ \notag
    &\quad\,
            -c_2s_0^{\frac{3-\gamma}{2}+\frac{q}{n}-\frac{p}{2n}[2(\kappa-1)+(1-\alpha)_{+}]}   
                         \sqrt{\psi_\alpha(s_0,t)}
  \end{align*}
for all $s_0\in(0,R^n)$ and $t\in(0,{\rm min}\{T,\tmax\})$.
Aided by Young's inequality, we can verify that for any $\eta>0$ there exists $c_3=c_3(\eta)>0$ such that
  \begin{align}\label{paphi2}
    \frac{\pa \phi}{\pa t}(s_0,t)
    &\ge c_1\psi_\alpha(s_0,t)-\eta\psi_\alpha(s_0,t)
            -c_2s_0^{\frac{2}{n}-\frac{p}{n}[(1-\alpha)_{+}+(\alpha-1)_{+}]}\psi_{\alpha}(s_0,t)
    \\ \notag
    &\quad\,
            -c_3\left(
                           s_0^{3-\gamma-\frac{4}{n}-\frac{p}{n}[2(m-1)_{+}+(1-\alpha)_{+}]}
                         +s_0^{3-\frac{2}{n}-\gamma}
                  \right.
    \\ \notag
    &\quad\,
                  \left.
                         +s_0^{\frac{4}{n}+1-\gamma-\frac{p}{n}[(1-\alpha)_{+}+2(\alpha-1)_{+}]}
                         +s_0^{3-\gamma+\frac{2q}{n}-\frac{p}{n}[2(\kappa-1)+(1-\alpha)_{+}]}
                  \right)
  \end{align}
for all $s_0\in(0,R^n)$ and $t\in(0,{\rm min}\{T,\tmax\})$.
Here, since it follows from the condition $1-\frac{2}{p}<\alpha<1+\frac{2}{p}$ that
  \[
    \frac{2}{n}-\frac{p}{n}[(1-\alpha)_{+}+(1-\alpha)_{+}]>\frac{2}{n}-\frac{2}{n}=0,
  \]
we can take $s_1\in(0,R^n)$ 
satisfying 
  \[
    s_1
    \le\left(\frac{c_1}{4c_2}\right)^\frac{1}{\frac{2}{n}-\frac{p}{n}[(1-\alpha)_{+}+(1-\alpha)_{+}]}.
  \]
Therefore we see that
  \begin{align}\label{paphi2third}
    s_0^{\frac{2}{n}-\frac{p}{n}[(1-\alpha)_{+}+(\alpha-1)_{+}]}\psi_{\alpha}(s_0,t)
    \le\frac{c_1}{4c_2}\psi_{\alpha}(s_0,t)
  \end{align}
for all $s_0\in(0,s_1)$ and $t\in(0,{\rm min}\{T,\tmax\})$.
By fixing $\eta=\frac{c_1}{4}$, we infer from \eqref{paphi2} and \eqref{paphi2third} that
  \begin{align}\label{paphi3}
    \frac{\pa \phi}{\pa t}(s_0,t)
    &\ge \frac{c_1}{2}\psi_\alpha(s_0,t)
    \\ \notag
    &\quad\,
            -c_3\left(
                           s_0^{3-\gamma-\frac{4}{n}-\frac{p}{n}[2(m-1)_{+}+(1-\alpha)_{+}]}
                         +s_0^{3-\frac{2}{n}-\gamma}
                  \right.
    \\ \notag
    &\quad\,
                  \left.
                         +s_0^{\frac{4}{n}+1-\gamma-\frac{p}{n}[(1-\alpha)_{+}+2(\alpha-1)_{+}]}
                         +s_0^{3-\gamma+\frac{2q}{n}-\frac{p}{n}[2(\kappa-1)+(1-\alpha)_{+}]}
                  \right)
  \end{align}
for all $s_0\in(0,s_1)$ and $t\in(0,{\rm min}\{T,\tmax\})$.
Next, putting
  \begin{align*}
    \theta_1&:={\rm max}\left\{
                                      \frac{4}{n}+\frac{p}{n}[2(m-1)_{+}+(1-\alpha)_{+}],
                                      \frac{2}{n},
                                      2-\frac{4}{n}+\frac{p}{n}[(1-\alpha)_{+}+2(\alpha-1)_{+}],
                               \right.
                               \\ 
                               &\quad\,
                               \left.
                                      -\frac{2q}{n}+\frac{p}{n}[2(\kappa-1)+(1-\alpha)_{+}]
                               \right\},
  \end{align*}
we show $\theta_1\in\left(0,2-\frac{p}{n}(1-\alpha)_{+}\right)$,
that is, we confirm that
  \begin{align}
  \label{theta1}
    \frac{4}{n}+\frac{p}{n}[2(m-1)_{+}+(1-\alpha)_{+}]<2-\frac{p}{n}(1-\alpha)_{+},
  \\ \label{theta2}
    \frac{2}{n}<2-\frac{p}{n}(1-\alpha)_{+},
  \\ \label{theta3}
    2-\frac{4}{n}+\frac{p}{n}[(1-\alpha)_{+}+2(\alpha-1)_{+}]<2-\frac{p}{n}(1-\alpha)_{+},
  \\ \label{theta4}
    -\frac{2q}{n}+\frac{p}{n}[2(\kappa-1)+(1-\alpha)_{+}]<2-\frac{p}{n}(1-\alpha)_{+}.
  \end{align}
Since the inequality
$2-\frac{4}{n}-\frac{p}{n}[2(m-1)_{+}+(1-\alpha)_{+}]>\frac{p}{n}(1-\alpha)_{+}$
holds from \eqref{lem4.1}, it follows 
that
  \[
    \left(2-\frac{p}{n}(1-\alpha)_{+}\right)
    -\left(\frac{4}{n}+\frac{p}{n}[2(m-1)_{+}+(1-\alpha)_{+}]\right)
    >0.
  \]
Moreover, we can establish 
from the condition $1-\frac{2}{p}<\alpha<1+\frac{2}{p}$ that
  \[
    \left(2-\frac{p}{n}(1-\alpha)_{+}\right)-\frac{2}{n}
    >\left(2-\frac{2}{n}\right)-\frac{2}{n}>0
  \]
and
  \begin{align*}
    \left(2-\frac{p}{n}(1-\alpha)_{+}\right)
    -\left(2-\frac{4}{n}+\frac{p}{n}[(1-\alpha)_{+}+2(\alpha-1)_{+}]\right) 
    &=\frac{4}{n}-\frac{2p}{n}[(1-\alpha)_{+}+(\alpha-1)_{+}]
  \\
    &>\frac{4}{n}-\frac{4}{n}=0.
  \end{align*}
Noticing from \eqref{lem4.1} that
$1>\frac{p}{n}[2(\kappa-1)+(1-\alpha)_{+}]-\frac{2q}{n}$,
we see from the condition $1-\frac{2}{p}<\alpha$ that
  \begin{align*}
    &\left(2-\frac{p}{n}(1-\alpha)_{+}\right)
    -\left(-\frac{2q}{n}+\frac{p}{n}[2(\kappa-1)+(1-\alpha)_{+}]\right)
  \\ &\quad\,
    >\left(2-\frac{2}{n}\right)-\left(\frac{p}{n}[2(\kappa-1)+(1-\alpha)_{+}]-\frac{2q}{n}\right)
  \\ &\quad\,
    >1-\left(\frac{p}{n}[2(\kappa-1)+(1-\alpha)_{+}]-\frac{2q}{n}\right)>0.
  \end{align*}
Thus, since we know that \eqref{theta1}--\eqref{theta4} hold, 
we have $\theta_1\in\left(0,2-\frac{p}{n}(1-\alpha)_{+}\right)$.
Invoking that $\gamma\in(0,1)$, $\theta_1\in\left(0,2-\frac{p}{n}(1-\alpha)_{+}\right)$ and $s_0<R^n$, we can deduce from \eqref{paphi3} that
there exists $c_4=c_4(R,m,\alpha,\kappa,p)>0$ such that
  \[
    \frac{\pa \phi}{\pa t}(s_0,t)
    \ge \frac{c_1}{2}\psi_\alpha(s_0,t)-c_3c_4s_0^{3-\gamma-\theta_1}
  \]
for all $s_0\in(0,s_1)$ and $t\in(0,{\rm min}\{T,\tmax\})$.
Moreover, in light of Lemma \ref{phipsi}, we can take $c_5>0$ such that
  \[
    \frac{\pa \phi}{\pa t}(s_0,t)
    \ge \frac{c_1}{2}c_5s_0^{\gamma-3+\frac{p}{n}(1-\alpha)_{+}}\phi^2(s_0,t)
          -c_3c_4s_0^{3-\gamma-\theta_1}
  \]
for all $s_0\in(0,s_1)$ and $t\in(0,{\rm min}\{T,\tmax\})$.
Hence we attain \eqref{keyineq}. 
%
%
As to (ii), since $m$ and $\alpha$ fulfill 
{\rm (A1)}, {\rm (A2)}, {\rm (A3)}, {\rm (B1-1)}, {\rm (B2)}, {\rm (C1-1)} or {\rm (C2)},
we can pick $\gamma\in(0,1)$ with \eqref{lem4.2}.
Applying Lemma \ref{I1}, Lemma \ref{I4}, part (ii) of Lemma \ref{I2} and  Lemma \ref{I3} 
to \eqref{paphi}, we obtain that there exist $c_6>0$ and $c_7>0$ such that
  \begin{align*}
    \frac{\pa \phi}{\pa t}(s_0,t)
    &\ge c_6\psi_\alpha(s_0,t)
           -c_7s_0^{3-\gamma-\frac{2}{n}-\frac{pm}{n}}
           -c_7s_0^{3-\frac{2}{n}-\gamma}
    \\ \notag
    &\quad\,
           -c_7s_0^{\frac{2}{n}+\frac{1-\gamma}{2}-\frac{p}{2n}[(1-\alpha)_{+}+2(\alpha-1)_{+}]}
                        \sqrt{\psi_{\alpha}(s_0,t)}
    \\ \notag
    &\quad\,
            -c_7s_0^{\frac{2}{n}-\frac{p}{n}[(1-\alpha)_{+}+(\alpha-1)_{+}]}\psi_{\alpha}(s_0,t)
    \\ \notag
    &\quad\,
            -c_7s_0^{\frac{3-\gamma}{2}+\frac{q}{n}-\frac{p}{2n}[2(\kappa-1)+(1-\alpha)_{+}]}   
                         \sqrt{\psi_\alpha(s_0,t)}
  \end{align*}
for all $s_0\in(0,R^n)$ and $t\in(0,{\rm min}\{T,\tmax\})$.
By an argument similar to the proof in (i), we can show that
there exists $c_8>0$ such that
  \begin{align*}
    \frac{\pa \phi}{\pa t}(s_0,t)
    &\ge \frac{c_6}{2}\psi_\alpha(s_0,t)
    \\ \notag
    &\quad\,
            -c_8\left(
                           s_0^{3-\gamma-\frac{2}{n}-\frac{pm}{n}}
                         +s_0^{3-\frac{2}{n}-\gamma}
                  \right.
    \\ \notag
    &\quad\,
                  \left.
                         +s_0^{\frac{4}{n}+1-\gamma-\frac{p}{n}[(1-\alpha)_{+}+2(\alpha-1)_{+}]}
                         +s_0^{3-\gamma+\frac{2q}{n}-\frac{p}{n}[2(\kappa-1)+(1-\alpha)_{+}]}
                  \right)
  \end{align*}
for all $s_0\in(0,s_2)$ and $t\in(0,{\rm min}\{T,\tmax\})$, 
where $s_2:=\left(\frac{c_6}{4c_7}\right)^\frac{1}{\frac{2}{n}-\frac{p}{n}[(1-\alpha)_{+}+(1-\alpha)_{+}]}$.
We set 
  \begin{align*}
    \theta_2&:={\rm max}\left\{
                                      \frac{2}{n}+\frac{pm}{n},
                                      \frac{2}{n},
                                      2-\frac{4}{n}+\frac{p}{n}[(1-\alpha)_{+}+2(\alpha-1)_{+}],
                               \right.
                               \\ 
                               &\quad\,
                               \left.
                                      -\frac{2q}{n}+\frac{p}{n}[2(\kappa-1)+(1-\alpha)_{+}]
                               \right\}.
  \end{align*}
Here, we note that \eqref{theta2}--\eqref{theta4} hold.
In order to verify that $\theta_2\in\left(0,2-\frac{p}{n}(1-\alpha)_{+}\right)$, we confirm that
  \begin{align}
    \label{theta5}
    \frac{2}{n}+\frac{pm}{n}<2-\frac{p}{n}(1-\alpha)_{+}.
  \end{align}
Since it follows from \eqref{lem4.2} that 
$2-\frac{2}{n}-\frac{pm}{n}
>\frac{p}{n}(1-\alpha)_{+}$,
we have that
  \[
    \left(2-\frac{p}{n}(1-\alpha)_{+}\right)-\left(\frac{2}{n}+\frac{pm}{n}\right)
    =\left(2-\frac{2}{n}-\frac{pm}{n}
\right)-\frac{p}{n}(1-\alpha)_{+}>0.
  \]
Thus, since we attain that \eqref{theta2}--\eqref{theta5} hold, 
we know that $\theta_2\in\left(0,2-\frac{p}{n}(1-\alpha)_{+}\right)$.
From the fact $s_0<R^n$ and Lemma \ref{phipsi} we can find $c_9>0$ and $c_{10}>0$ such that 
  \[
    \frac{\pa \phi}{\pa t}(s_0,t)
    \ge \frac{c_6}{2}c_9s_0^{\gamma-3+\frac{p}{n}(1-\alpha)_{+}}\phi^2(s_0,t)
          -c_8c_{10}s_0^{3-\gamma-\theta_2}
  \]
for all $s_0\in(0,s_2)$ and $t\in(0,{\rm min}\{T,\tmax\})$, which concludes the proof.
\end{proof}
%
\section{Proof of the main results}
Now we are in a position to complete the proofs of Theorems \ref{mainthm1} and \ref{mainthm2}.
By the similar argument as that in \cite[Lemma 4.1]{B-F-L} we can proved the following lemma which is needed for the proof of Theorem \ref{mainthm1}.
%
%
%
\begin{lem}\label{phi0}
Let $\gamma\in(0,1)$, $s_0\in(0,R^n)$, $M_1\ge0$ and $\eta\in(0,1)$ and 
set $s_\eta:=(1-\eta)s_0$ as well as $r_1:=s_\eta^\frac{1}{n}$.
If 
  \[
    \int_{B_{r_1}(0)} u_0\ge M_1,
  \]
then 
  \[
    \phi(s_0,0)\ge\frac{\eta^2M_1}{\omega_{n-1}}\cdot s_0^{2-\gamma}.
  \]
\end{lem}
\begin{proof}[Proof of Theorem \ref{mainthm1}]
We have from Lemma \ref{keylem}
 that \eqref{keyineq} holds.
By Lemma \ref{phi0} and 
an argument similar to that in the proof of \cite[Theorem 1.1]{B-F-L}
we see that $\tmax<T$.
Thanks to Lemma \ref{solution}, we arrive at the conclusion \eqref{blowup}.
\end{proof}
%
%
Before giving the proof of Theorem \ref{mainthm2}, we show the pointwise upper estimate for $u$.
\begin{lem}\label{lemupower}
Let $\Omega=B_R(0)\subset\Rn\ (n\ge3)$ with $R>0$ 
and let $\chi>0$, $\kappa\ge1$, $\mu_1>0$, $q\ge0$, 
$M_0>0$, $T>0$ and $L>0$.
Suppose that $\lambda$ and $\mu$ satisfy \eqref{lm} and \eqref{mu}.
Assume that $m>0$ and $\alpha>0$ fulfill that
  \[
    m\ge1
    \quad\mbox{and}\quad
    m-\alpha\in\left(-\frac{1}{n},\frac{n-2}{n}\right].
  \]
For all $\varepsilon>0$ set $p:=\frac{n(n-1)}{(m-\alpha)n+1}+\varepsilon$.
Then there exists $C>0$ such that the following property 
holds\/{\rm :}
  If $u_0$ satisfies 
\eqref{initial} and $\int_\Omega u_0=M_0$
  as well as
    \begin{align}
      u_0(x)\le L|x|^{-p}\quad\mbox{for all}\ x\in\Omega
    \end{align}
  and 
    $(u,v)\in\left(C^0(\overline{\Omega}\times[0,T))\cap
                       C^{2,1}(\overline{\Omega}\times(0,T))\right)^2$
  is a classical solution to \eqref{PE},
  then
    \begin{align}\label{upower}
      u(x,t)\le C|x|^{-p}\quad\mbox{for all}\ x\in\Omega\ \mbox{and}\ t\in(0,T).
    \end{align}
\end{lem}
%
\begin{proof}
We have that there exists $\lambda_1>0$ such that $\lambda(|x|)\le\lambda_1$ for all $x\in\overline{\Omega}$. 
Here, we put 
  $\tilde{u}(x,t):=e^{-\lambda_1 t}u(x,t)$, 
  $D(x,t,\rho):=m(e^{\lambda_1 t}\rho+1)^{m-1}$ and
  $S(x,t,\rho):=-\chi(e^{\lambda_1 t}\rho+1)^{\alpha-1}\rho$.
Then we obtain from \eqref{PE} that
  \begin{align}
    \begin{cases}
      \tilde{u}_t\le\nabla\cdot(D(x,t,\tilde{u})\nabla\tilde{u}+S(x,t,\tilde{u})\nabla v)
      \quad&\mbox{in}\ \Omega\times(0,T),
    \\
      (D(x,t,\tilde{u})\nabla\tilde{u}+S(x,t,\tilde{u})\nabla v)\cdot\nu=0
      \quad&\mbox{on}\ \pa\Omega\times(0,T),
    \\
      \tilde{u}(\cdot,0)=u_0
      \quad&\mbox{in}\ \Omega.
    \end{cases}
  \end{align}
Moreover, we can verify that
  \begin{align*}
    D(x,t,\rho)&\ge m\rho^{m-1}, 
  \\ 
    D(x,t,\rho)&\le m(e^{\lambda_1 T}+1)^{m-1}\max\{\rho,1\}^{m-1},
  \\
    |S(x,t,\rho)|&\le \chi(e^{\lambda_1 T}\rho+1)^\alpha
                    \le \chi(e^{\lambda_1 T}+1)^\alpha\max\{\rho,1\}^\alpha
  \end{align*}
for all $x\in\Omega$, $t\in(0,T)$ and $\rho\in(0,\infty)$ 
and 
  \[
    \int_\Omega \tilde{u}(\cdot,0)=M_0. 
  \]
Now we take $\theta>n$ fulfilling that 
  \[
    m-\alpha\in\left(\frac{1}{\theta}-\frac{1}{n},\frac{1}{\theta}+\frac{n-2}{n}\right]
  \]
and
  \[
    p=\frac{n(n-1)}{(m-\alpha)n+1}+\varepsilon
      >\frac{n(n-1)}{(m-\alpha)n+1-\frac{n}{\theta}}
      =\frac{(n-1)}{m-\alpha+\frac{1}{n}-\frac{1}{\theta}}.
  \]
By an argument similar to that in the proof of \cite[Lemma 5.2]{B-F-L}, we have that
  \[
    \int_{\Omega}|x|^{(n-1)\theta}|\nabla v(x,t)|^\theta\, dx
    \le \left(\frac{2e^{\lambda_1 T}M_0}{\omega_{n-1}}\right)^\theta|\Omega|
  \]
for all $t\in(0,T)$.
Thus, from \cite[Theorem 1.1]{F-2020_profiles} we obtain that there exists $c_1>0$ such that
$\tilde{u}(x,t)\le c_1|x|^{-p}$ for all $x\in\Omega$ and $t\in(0,T)$, which implies \eqref{upower}.
\end{proof}
%
\begin{proof}[Proof of Theorem \ref{mainthm2}]
We set $p_0:=\frac{n(n-1)}{(m-\alpha)n+1}$.
Here, we note from (E1), (F1) and (F2) that $m-\alpha<\frac{n-2}{n}$ and $p_0>n$. 
In the case $n\ge3$, by a direct computation we infer from the conditions
$\alpha<\frac{2}{n+1}m+\frac{n^2-n+2}{n(n+1)}$,
$\alpha<-\frac{1}{n-2}m+\frac{n^2-2}{n(n-2)}$ and 
$m-\alpha<\frac{n-2}{n}$
that
  \begin{align}\label{condip0}
    \alpha<1+\frac{2}{p_0},\quad
    m<1+\frac{n-2}{p_0}\quad\mbox{and}\quad
    m-\alpha<\frac{n-2}{p_0}.
  \end{align}
In the case (i), noting that (E1) yields \eqref{condip0},
we can pick $\varepsilon_1>0$ so small that
  \begin{align}\label{condip1}
    \alpha<1+\frac{2}{p_1},\quad 
    1\le m<1+\frac{n-2}{p_1},\quad
    m-\alpha<\frac{n-2}{p_1}
  \end{align}
and
  \begin{align}\label{kappaIV}
    \kappa<1+\frac{n-2}{p_1}+\frac{q}{p_1}-(m-1)-(1-\alpha)_{+},
  \end{align}
where $p_1:=p_0+\varepsilon_1$.
We take $L>0$ and $T>0$.
Moreover, we choose $r_1\in(0,R)$ and $u_0\in C^0(\overline{\Omega})$ with \eqref{initial} 
satisfying $\int_\Omega u_0=M_0$ and $\int_{B_{r_1}(0)} u_0 \ge M_1$ as well as $u_0(x)\le L|x|^{-p_1}$ for all $x\in\Omega$.
By virtue of Lemma \ref{lemupower} we can find $C>0$ complying with \eqref{upower}.
Thus, noticing from \eqref{condip1} and \eqref{kappaIV} that (A4), (B3) and 
{\rm ({\bf I\hspace{-.1em}V})} hold, we see from Theorem \ref{mainthm1} that the solution $(u,v)$ blows up in finite time.
On the other hand, in the case (ii) we obtain 
from the conditions 
$-\frac{2}{n-3}m+\frac{n^2-n-2}{n(n-3)}<\alpha$,
$\alpha<-\frac{n+2}{n-4}m+\frac{2n^2-n-4}{n(n-4)}$,
$\alpha\le\frac{n+2}{3}m-\frac{n^2-4}{3n}$
and $-\frac{n+2}{n-4}m+\frac{2n^2-n-4}{n(n-4)}\le\alpha$ that
  \[
    1-\frac{2}{p_0}<\alpha,\quad
    m<1+\frac{n-4}{2p_0},\quad
    2m-\alpha\le1+\frac{n-4}{p_0}
  \]
and
  \[
    1+\frac{n-4}{2p_0}\le m.
  \]
Therefore, by picking $\varepsilon_2>0$ small enough, 
we have from (F1) and (ii) that $m$ and $\alpha$ fulfill 
that
  \begin{align}\label{condip21}
    1-\frac{2}{p_2}<\alpha<1+\frac{2}{p_2},\quad
    1\le m<1+\frac{n-4}{2p_2},\quad
    2m-\alpha\ge1+\frac{n-4}{p_2}
  \end{align}
or 
  \begin{align}\label{condip22}
    1-\frac{2}{p_2}<\alpha<1+\frac{2}{p_2},\quad
    1+\frac{n-4}{2p_2}\le m<1+\frac{n-2}{p_2},\quad
    m-\alpha<\frac{n-2}{p_2}
  \end{align}
and $\kappa$ satisfies that
  \begin{align}\label{kappaIV2}
    \kappa<1+\frac{n-2}{p_2}+\frac{q}{p_2}-(m-1)-(1-\alpha)_{+},
  \end{align}
where $p_2:=p_0+\varepsilon_2$.
Moreover, in the case (iii) we obtain that
  \begin{align}\label{condip3}
    1-\frac{2}{p_3}<\alpha<1+\frac{2}{p_3},\quad
    1\le m<1+\frac{n-4}{2p_3},\quad
    2m-\alpha<1+\frac{n-4}{p_3}
  \end{align}
and
  \begin{align}\label{kappaII}
    \kappa<1+\frac{n}{2p_3}+\frac{q}{p_3}-\frac{(1-\alpha)_{+}}{2},
  \end{align}
where $p_3:=p_0+\varepsilon_3$ with some $\varepsilon_3>0$.
Accordingly, since we can verify from \eqref{condip21}, \eqref{condip22} and \eqref{kappaIV2} that (C3), (D2) and {\rm ({\bf I\hspace{-.1em}V})} hold and 
we can confirm from \eqref{condip3} and \eqref{kappaII} that
(C1), (D1) and {\rm ({\bf I\hspace{-.1em}I})} hold, we arrive at the conclusion by an argument similar to that in the proof of the case (i).
\end{proof}

%
\smallskip
\section*{Acknowledgments}
The author would like to thank Professor Tomomi Yokota for his encouragement and comments on the manuscript.

%

\bibliographystyle{plain}

\begin{thebibliography}{10}

\bibitem{B-B-T-W}
N.~Bellomo, A.~Bellouquid, Y.~Tao, and M.~Winkler.
\newblock Toward a mathematical theory of {K}eller--{S}egel models of pattern
  formation in biological tissues.
\newblock {\em Math. Models Methods Appl. Sci.}, 25:1663--1763, 2015.

\bibitem{B-F-L}
T.~Black, M.~Fuest, and J.~Lankeit.
\newblock Relaxed parameter conditions for chemotactic collapse in
  logistic-type parabolic--elliptic {K}eller--{S}egel systems.
\newblock {\em preprint, arXiv:2005.12089[math.AP]}.

\bibitem{C-S-2012}
T.~Cie\'{s}lak and C.~Stinner.
\newblock Finite-time blowup and global-in-time unbounded solutions to a
  parabolic--parabolic quasilinear {K}eller--{S}egel system in higher
  dimensions.
\newblock {\em J. Differential Equations}, 252:5832--5851, 2012.

\bibitem{C-S-2015}
T.~Cie\'{s}lak and C.~Stinner.
\newblock New critical exponents in a fully parabolic quasilinear
  {K}eller--{S}egel system and applications to volume filling models.
\newblock {\em J. Differential Equations}, 258:2080--2113, 2015.

\bibitem{C-W}
T.~Cie\'{s}lak and M.~Winkler.
\newblock Finite-time blow-up in a quasilinear system of chemotaxis.
\newblock {\em Nonlinearity}, 21:1057--1076, 2008.

\bibitem{F-2020_profiles}
M.~Fuest.
\newblock Blow-up profiles in quasilinear fully parabolic {K}eller--{S}egel
  systems.
\newblock {\em Nonlinearity}, 33:2306--2334, 2020.

\bibitem{F_2021_optimal}
Mario Fuest.
\newblock Approaching optimality in blow-up results for {K}eller-{S}egel
  systems with logistic-type dampening.
\newblock {\em NoDEA Nonlinear Differential Equations Appl.}, 28(2):16, 2021.

\bibitem{H-P}
T.~Hillen and K.J. Painter.
\newblock A user's guide to {PDE} models for chemotaxis.
\newblock {\em J. Math. Biol.}, 58:183--217, 2009.

\bibitem{I-S-Y}
S.~Ishida, K.~Seki, and T.~Yokota.
\newblock Boundedness in quasilinear {K}eller--{S}egel systems of
  parabolic--parabolic type on non-convex bounded domains.
\newblock {\em J. Differential Equations}, 256:2993--3010, 2014.

\bibitem{J-L}
W.~J\"{a}ger and S.~Luckhaus.
\newblock On explosions of solutions to a system of partial differential
  equations modelling chemotaxis.
\newblock {\em Trans. Amer. Math. Soc.}, 329:819--824, 1992.

\bibitem{K-S}
E.F. Keller and L.A. Segel.
\newblock Initiation of slime mold aggregation viewed as an instability.
\newblock {\em J. Theoret. Biol.}, 26:399--415, 1970.

\bibitem{L-2020}
J.~Lankeit.
\newblock Infinite time blow-up of many solutions to a general quasilinear
  parabolic--elliptic {K}eller--{S}egel system.
\newblock {\em Discrete Contin. Dyn. Syst. Ser. S}, 13:233--255, 2020.

\bibitem{L-W}
J.~Lankeit and M.~Winkler.
\newblock Facing low regularity in chemotaxis systems.
\newblock {\em Jahresber. Dtsch. Math.-Ver.}, 122:35--64, 2020.

\bibitem{O-T-Y-M}
K.~Osaki, T.~Tsujikawa, A.~Yagi, and M.~Mimura.
\newblock Exponential attractor for a chemotaxis-growth system of equations.
\newblock {\em Nonlinear Anal.}, 51:119--144, 2002.

\bibitem{Tanaka-Y_2020}
Y.~Tanaka and T.~Yokota.
\newblock Blow-up in a parabolic--elliptic {K}eller--{S}egel system with
  density-dependent sublinear sensitivity and logistic source.
\newblock {\em Math. Methods Appl. Sci.}, 43:7372--7396, 2020.

\bibitem{T-W}
Y.~Tao and M.~Winkler.
\newblock Boundedness in a quasilinear parabolic--parabolic {K}eller--{S}egel
  system with subcritical sensitivity.
\newblock {\em J. Differential Equations}, 252:692--715, 2012.

\bibitem{Tello-W-2007}
J.I. Tello and M.~Winkler.
\newblock A chemotaxis system with logistic source.
\newblock {\em Comm. Partial Differential Equations}, 32:849--877, 2007.

\bibitem{W-2010}
M.~Winkler.
\newblock Boundedness in the higher-dimensional parabolic--parabolic chemotaxis
  system with logistic source.
\newblock {\em Comm. Partial Differential Equations}, 35:1516--1537, 2010.

\bibitem{W_2010_quasilinear}
M.~Winkler.
\newblock Does a `volume-filling effect' always prevent chemotactic collapse?
\newblock {\em Math. Methods Appl. Sci.}, 33:12--24, 2010.

\bibitem{W-2011}
M.~Winkler.
\newblock Blow-up in a higher-dimensional chemotaxis system despite logistic
  growth restriction.
\newblock {\em J. Math. Anal. Appl.}, 384:261--272, 2011.

\bibitem{W-2018}
M.~Winkler.
\newblock Finite-time blow-up in low-dimensional {K}eller--{S}egel systems with
  logistic-type superlinear degradation.
\newblock {\em Z. Angew. Math. Phys.}, 69:Art. 69, 40, 2018.

\bibitem{W-2019}
M.~Winkler.
\newblock Global classical solvability and generic infinite-time blow-up in
  quasilinear {K}eller--{S}egel systems with bounded sensitivities.
\newblock {\em J. Differential Equations}, 266:8034--8066, 2019.

\bibitem{W-D}
M.~Winkler and K.C. Djie.
\newblock Boundedness and finite-time collapse in a chemotaxis system with
  volume-filling effect.
\newblock {\em Nonlinear Anal.}, 72:1044--1064, 2010.

\bibitem{Z-2015}
J.~Zheng.
\newblock Boundedness of solutions to a quasilinear parabolic--elliptic
  {K}eller--{S}egel system with logistic source.
\newblock {\em J. Differential Equations}, 259:120--140, 2015.

\bibitem{Z-2017}
J.~Zheng.
\newblock A note on boundedness of solutions to a higher-dimensional
  quasi-linear chemotaxis system with logistic source.
\newblock {\em ZAMM Z. Angew. Math. Mech.}, 97:414--421, 2017.

\end{thebibliography}

\end{document}